\documentclass{amsart}
\usepackage{amsmath}
\usepackage{amssymb}
\usepackage{amsthm}
\usepackage{mathscinet}
\usepackage{tikz-cd}
\usepackage{adjustbox}

\usepackage[
backend=biber,
style=alphabetic,
sorting=ynt
]{biblatex}

\usepackage{xcolor}

\usepackage[margin=1in]{geometry}

\newcommand{\norm}[1]{\left\Vert#1\right\Vert}
\newcommand{\abs}[1]{\left\vert#1\right\vert}

\newcommand{\ip}[1]{\left\langle#1\right\rangle}
\newcommand{\p}[1]{\left(#1\right)}

\title{Fourier series for singular measures in higher dimensions}
\author{Chad Berner, John E. Herr, Palle E.T. Jorgensen, and Eric S. Weber}

\newtheorem{thmalpha}{Theorem}
\newtheorem{thm}{Theorem}[section]
\newtheorem{prop}[thm]{Proposition}
\newtheorem{lem}[thm]{Lemma}
\theoremstyle{definition}
\newtheorem{defn}[thm]{Definition}
\newtheorem{cor}{Corollary}[section]
\newtheorem{ex}[thm]{Example}
\newtheorem{rem}[thm]{Remark}

\newtheorem*{thm*}{Theorem}
\newtheorem*{lem*}{Lemma}

\begin{document}

\begin{abstract}
For multi-variable finite measure spaces, we present in this paper a new framework for non-orthogonal $L^2$ Fourier expansions. Our results hold for probability measures $\mu$ with finite support in $\mathbb{R}^d$ that satisfy a certain disintegration condition that we refer to as ``slice-singular''. In this general framework, we present explicit $L^{2}(\mu)$-Fourier expansions, with Fourier exponentials having positive Fourier frequencies in each of the d coordinates. Our Fourier representations apply to every  $f \in L^2(\mu)$, are based on an extended Kaczmarz algorithm, and use a new recursive $\mu$ Rokhlin disintegration representation. In detail, our Fourier series expansion for $f$ is in terms of the multivariate Fourier exponentials $\{e_n\}$, but the associated Fourier coefficients for $f$ are now computed from a Kaczmarz system $\{g_n\}$ in $L^{2}(\mu)$ which is dual to the Fourier exponentials.  The $\{g_n\}$ system is shown to be a Parseval frame for $L^{2}(\mu)$. Explicit computations for our new Fourier expansions entail a detailed analysis of subspaces of the Hardy space on the polydisk, dual to $L^{2}(\mu)$, and an associated d-variable Normalized Cauchy Transform. Our results extend earlier work for measures $\mu$ in one and two dimensions, i.e., $d=1 (\mu $ singular), and $d=2 (\mu$ assumed slice-singular). Here our focus is the extension to the cases of measures $\mu$ in dimensions $d >2$. Our results are illustrated with the use of explicit iterated function systems (IFSs), including the IFS generated Menger sponge for $d=3$.
\end{abstract}

\maketitle

\section{Introduction}

We are interested in constructing Fourier expansions for elements in $L^2(\mu)$, where $\mu$ is a singular Borel probability measure on $\mathbb{R}^{d}$.  Substantial work has focused on two classes, the case of (i) orthogonal expansions (analyzed for Cantor measures and related IFS measures), and (ii) the case of frame-like expansions. In both cases (for certain classes of fractal measures), there now exist explicit algorithms for analysis and synthesis for the corresponding Hilbert space $L^2(\mu)$ \cite{Jorgensen1998Dense,Herr2017Fourier}.  By this we mean that there exists a sequence of frequencies $\{\lambda_{n}\}$ such that for every $f \in L^2(\mu)$, 
\begin{equation} \label{Eq:fourier-series}
    f(x) = \sum_{n=0}^{\infty} c_{n} e^{2 \pi i \lambda_{n} \cdot x}.
\end{equation}
The foundational result establishing the existence of orthogonal expansions for certain singular measures appeared in \cite{Jorgensen1998Dense}.  Non-orthogonal expansions using the Kaczmarz algorithm were introduced in \cite{Herr2017Fourier}--notably, the construction therein is valid for any singular measure but only works for $d=1$.  Recently, we proved in \cite{Herr2022Fourier} that non-orthogonal Fourier series expansions of the form in Equation \eqref{Eq:fourier-series} could be obtained when $d=2$ provided that the measure possessed what we called the ``slice singular'' property.  Not every singular measure on $\mathbb{R}^2$ is slice-singular, but we demonstrated that large classes of affine fractal measures are slice-singular.  The slice singular property is defined in terms of the Rokhlin decomposition of measures (see Subsection \ref{ssec:rokhlin} and Definition \ref{D:slicesingular} for details).


The purpose of this paper is twofold.  The primary purpose is to extend the results of \cite{Herr2022Fourier} to dimensions $d \geq 3$.  This requires a careful generalization of the definition of slice singular to higher dimensions, and a similarly careful analysis of the operator version of the Kaczmarz algorithm we used in the $d=2$ case.  Our main result can be stated as follows:
\begin{thmalpha} \label{Th:main1}
    If $\mu$ is an $x_{d}$ slice singular Borel probability measure on $[0,1)^{d}$, then for all $f\in L^{2}(\mu)$, there is a $d$-indexed sequence $\{c_{n_{1},\dots, n_{d}}\}$ such that
    \begin{equation}\label{MRi}f(x_{1},\dots, x_{d})=\sum_{n_{d}=0}^{\infty}\dots \sum_{n_{1}=0}^{\infty}c_{n_{1},\dots, n_{d}}e^{2\pi i (n_{1}x_{1}+\dots +n_{d}x_{d})}\end{equation}
    where the limits are in norm and taken in order starting with the right most series.

    Furthermore, for each $n_{1},\dots ,n_{d}$, the mapping of $f\to c_{n_{1},\dots, n_{d}}$ is a continuous linear functional, and the mapping
    $f\to \{c_{n_{1},\dots, n_{d}}\}$ is an isometry into $\ell^{2}(\mathbb{N}^{d})$.
\end{thmalpha}
We will construct the linear functionals explicitly in Subsection \ref{ssec:coeff}, and we show the existence of slice-singular measures in Subsection 3.4.

The secondary purpose of the paper is to revisit the results of \cite{Herr2022Fourier} in the $d=2$ case.  We will utilize the results we obtain here for the $d\geq3$ case--our higher dimensional results shed new light on the nature of the Fourier series results for $d=2$ along with related constructs in harmonic analysis and operator theory.  Indeed, reverting back to the $d=1$ case, we emphasize that the existence of Fourier series there \cite{Pol93,Herr2017Fourier} is intimately related to classical results in the Hardy space $H^2(\mathbb{D})$, particularly the model subspaces of the Hardy space \cite{nikolski1986treatise,Sarason1994Sub-Hardy,cima2000backward} and the Normalized Cauchy Transform (NCT) originally introduced by Clark \cite{Clark72,Aleks89a}.
By utilizing the techniques we develop for proving Theorem \ref{Th:main1}, we will be able to provide a definition for a NCT that maps $L^2(\mu)$ isometrically into the Hardy space of the bidisk $H^2(\mathbb{D}^2)$.  Our construction also allows us to provide concrete representations for de Branges spaces \cite{DeBranges1960Some,dB61a} as subspaces of $H^2(\mathbb{D}^2)$ that correspond to the image of the NCT as we define it.

The content of Section \ref{sec:d-dim} is the proof of Theorem \ref{Th:main1}, which utilizes a direct integral decomposition of certain spaces in order to analyze the spectral theory of operators that result from the operator valued Kaczmarz algorithm.  These operators also arise from work of de Branges \cite[Lemma 2]{DeBranges1960Some} (stated in Lemma \ref{L:deBranges2} for reference), which plays a crucial role in the analysis.  Through our viewpoint of the Kaczmarz algorithm as well as the slice singular property of our measures, we are able to provide a complete description of the Fourier series coefficients that appear in Equation \eqref{MRi}.

After completing our proof of the existence of d-dimensional Fourier series expansions for slice singular measures (with $d \geq3)$, we return to the case of $d=2$ in Section \ref{sec:NCT}.  Using the insight from the de Branges Lemma and the Kaczmarz algorithm, we are able to define NCTs in two (and higher) dimensions.  Unlike the $d=1$ case, the NCTs are not uniquely defined, since they depend on the Rokhlin decomposition of the measure.  In fact, when the measure is slice singular in multiple directions, we generally obtain several different NCTs associated to the different decompositions.  In Section \ref{sec:NCT}, we define and discuss properties of the NCTs as well as their images.  We also prove when the several transforms are in fact identical, and when their images coincide.  Our main result in the section is Theorem \ref{Th:NCT-image} in which we describe the image of the NCT as a vector-valued de Branges' space.

Our paper concludes with a formalization of a Normalized Cauchy Transform for the arbitrary d-dimensional case of slice-singular measures. We discuss the image of this transform and some of its properties relating to properties of the underlying measure

As we previously mentioned, the study of Fourier series for singular measures has a substantial history.  To place our contribution within the existing theory, we recall some of that history here.

Much of the story starts with measures $\mu$ that are spectral, meaning there is a sequence of exponential functions that form an orthonormal basis of $L^{2}(\mu)$. Jorgensen and Pedersen \cite{Jorgensen1998Dense} showed that there exists singular measures that are spectral. Even more surprising, they also proved that the uniform measure on the Cantor middle third set is not spectral.

Naturally, the study of spectral measures led to understanding Fourier frames for singular measures, which has been explored in variety of different ways for decades. Recall that a sequence of vectors $\{f_{n}\}_{n=0}^{\infty}$ in a Hilbert space $H$ is called a frame if there exists $A,B>0$ such that
$$A||f||^{2}\leq \sum_{n=0}^{\infty}|\langle f, f_{n}\rangle|^{2}\leq B||f||^{2}$$ for all $f\in H$. If $A=B=1$, then $\{f_{n}\}_{n=0}^{\infty}$ is called a Parseval frame.  He, Lai, and Lau \cite{He2013Exponential} proved that a measure admitting a Fourier frame is of pure type. Namely, the measure must be either absolutely continuous or singular with respect to Lebesgue measure. Similarly, Dutkay and Lai \cite{DL14a} prove that measures with Fourier frames must satisfy a certain uniformity condition. Examples of singular measures admitting Fourier frames have been constructed in the Cantor-4 set \cite{Picioroaga2017Fourier}.

However, the question of exactly when a singular measure $\mu$ admits a Fourier frame for $L^{2}(\mu)$ is still open. Moreover, some studies also shifted to exploring a weaker condition than a singular measure admitting a Fourier frame, namely when does a measure $\mu$ possess the property that all $f\in L^{2}(\mu)$ can be expressed as a Fourier series whose coefficients are obtained by bounded linear functionals, which is weaker than $\mu$ having a frame of exponential functions.  Specifically, rather than asking whether $\mu$ has a frame of exponentials, we can consider a ``pseudodual'' condition (introduced by Li and Ogawa \cite{Li2001Pseudo-duals} in the context of wavelets) where there exists a sequence of frequencies $\{\lambda_{n}\}$ \emph{and} a sequence $\{ h_{n} \} \subset L^2(\mu)$ such that
\[ f = \sum_{n} \langle f, h_{n} \rangle e^{2 \pi i \lambda_{n} \cdot x}. \]

This psuedodual question was in fact answered in the one dimensional case in \cite{Herr2017Fourier} using the Kaczmarz algorithm and an effectiveness theorem for stationary sequences by Kwapien and Mycielski \cite{Kwapien2001Kaczmarz}. Later this work continued in two dimensions in \cite{Herr2022Fourier} where singular measure are considered to have a stronger property called ``slice singular" so that they were able to provide Fourier expansions in this setting where they relied on an operator version of the Kaczmarz algorithm.

\section{Preliminaries}

In this section, we begin by using the Rokhlin disintegration theorem to define slice-singular measures. Then we review the Fourier expansions for singular measures on the torus using the Kaczmarz algorithm. Finally, we close this section by recalling some theorems of Herglotz and de Branges, and we define direct integrals. 

\subsection{Slice-singular measures} \label{ssec:rokhlin}

Our presentation of Fourier expansions in higher dimensions requires a stronger notation of singular, similarly with the two dimensional case \cite{Herr2022Fourier}. In our framework for our Fourier expansions, we require the measure to be what we will call \textbf{slice singular} that we define using the Rokhlin Disintegration Theorem \cite{Rohlin1949Fundamental,Rohlin1949Decomposition}, which is a generalization of the theorem of Fubini-Tonelli.

\begin{thm*}[Rokhlin Disintegration]
For a Borel probability measure $\mu$ on metric space $A\times B$, there is a unique Borel probability measure $\sigma$ on $B$, namely, 
$\sigma=\mu \circ \pi_{B}^{-1}$ where $\pi_{B}: A\times B\to B$ is the projection onto $B$
and a $\sigma$ almost everywhere uniquely determined family of Borel probability measures $\{\gamma^{b}\}_{b\in B}$ on $A$ such that the following:
\begin{enumerate}
    \item If $f\in L^{1}(\mu)$, for $\sigma$ almost every $b$, $f(a,b)\in L^{1}(\gamma^{b})$, and $\int_{A}f(a,b)d\gamma^{b}\in L^{1}(\sigma).$
    \item For each $f\in L^{1}(\mu)$,
   $$\int_{A\times B}fd\mu=\int_{B}\int_{A}f(a,b)d\gamma^{b}d\sigma.$$
\end{enumerate}
\end{thm*}

\begin{defn} \label{D:slicesingular}
For a Borel probability measure $\mu$ on $[0,1)^{d-1}\times [0,1)$ where $d\geq 2$, by the Rokhlin Disintegration Theorem, there is Borel probability measure $\mu_{1}$ on $[0,1)^{d-1}$ that we will call the \textbf{marginal measure} and family of Borel probability measures $\{\gamma^{x_{1},\dots, x_{d-1}}\}$ on $[0,1)$ indexed by $[0,1)^{d-1}$ that we will call the \textbf{slice measures} of $\mu$ such that the hypothesis of Rokhlin's theorem hold.
We can also disintegrate analogously on other variables.

Now we define a sequence of measures for convenience of notation. Let $\mu_{0}=\mu$ and for $n\geq 1$, let $\mu_{n}$ be the $d-n$ dimensional marginal measure from $\mu_{n-1}$ in the Rokhlin disintegration where the one dimensional slice measures from $\mu_{n-1}$ are taken from the right most variable. 
Furthermore, for each $k$, let 
$\{\gamma^{x_{1},\dots x_{k}}\}$ be the one dimensional slice measures of $\mu_{d-1-k}$.
We will use this sequence of measures later in section 3 to discuss where the coefficients from Theorem A come from.

We inductively define $\mu$ to be \textbf{$x_{d}$ slice singular} if $\mu_{1}$ is $x_{d-1}$ slice singular and for $\mu_{1}$ almost every point $(x_{1},\dots, x_{d-1})$, $\gamma^{x_{1},\dots, x_{d-1}}$ is singular.
If $d=2$, then we say $\mu$ is $x_{2}$ slice singular if the family $\{\gamma^{x_{1}}\}$ is singular for $\mu_{1}$ almost every $x_{1}$ and $\mu_{1}$ is singular.
\end{defn}
Again, we will begin the discussion in the case of peeling back the family of measures from the last variable to the first just for simplicity, but a definition where a family of one dimensional measures comes from another variable in any step is analogous and will be discussed later. It is important that the family of measures from the disintegration is one dimensional, which we will attempt to make clear soon.

See \cite{bezuglyi2019graph,bezuglyi2021symmetric,chang1997conditioning} for further applications of disintegration theory in probability and harmonic analysis.

\subsection{Fourier series on the torus}

However, in order to understand the case of higher dimensions, we need to understand the one dimensional case \cite{Herr2017Fourier,herr2020harmonic}.
In that paper, it was shown that for a singular Borel probability measure $\mu$ on $[0,1)$, the sequence $\{e^{2\pi i nx}\}_{n=0}^{\infty}$ is a complete stationary sequence in $L^{2}(\mu)$, and that every element of $L^{2}(\mu)$ can be expressed as a Fourier series using the Kaczmarz algorithm, which we define from Haller and Szwarc in \cite{Haller2005Kaczmarz}.
\begin{defn}\label{D:aux}
Given a complete sequence $\{e_{n}\}_{n=0}^{\infty}$ in a Hilbert space $H$,
Define the following \textbf{auxiliary sequence} in $H$ recursively:
$$g_{0}=e_{0}$$
\begin{equation}\label{Eq:aux}
g_{n}=e_{n}-\sum_{k=0}^{n-1}\langle e_{n},e_{k}\rangle g_{k}.
\end{equation}

Then $\{e_{n}\}$ is \textbf{effective} if for every
$f\in H$,
$$\sum_{k=0}^{n}\langle f,g_{k}\rangle e_{k}\to f$$ in $H$.
\end{defn}

The $\{e_{n}\}$ being effective is also equivalent to $\{g_{n}\}$ being a Parseval frame by \cite{Haller2005Kaczmarz}.
In \cite{Herr2017Fourier}, it was shown that $\{e^{2\pi i nx}\}_{n=0}^{\infty}$ is effective in $L^{2}(\mu)$ where $\mu$ is any singular Borel probability measure on $[0,1)$. Throughout the rest of the paper when we refer to an \textbf{auxilary sequence}, it will refer to the sequence generated from $\{e^{2\pi i nx}\}_{n=0}^{\infty}$ in $L^{2}(\mu)$ where $\mu$ is a Borel probability measure.

\subsection{Some useful theorems}
Recall the Hardy space we will define as follows:
$$H^{2}(\mathbb{D})=\{w\to \sum_{n=0}^{\infty}c_{n}w^{n}: \sum_{n=0}^{\infty}|c_{n}|^{2}<\infty\}.$$
Furthermore, an \textbf{inner function} $b(w)$ is a bounded analytic function on $\mathbb{D}$ such that 
$$||b(w)f(w)||_{H^{2}(\mathbb{D})}=||f(w)||_{H^{2}(\mathbb{D})}$$ for any $f\in H^{2}(\mathbb{D})$.

The following statements will be proven useful for us later in the discussion:

\begin{thm*}[Herglotz]
There is a one-to-one correspondence between singular Borel probability measures on $[0,1)$ $\mu$, and inner functions $b(w)$ on $\mathbb{D}$ such that $b(0)=0$ where
$$\frac{1+b(w)}{1-b(w)}=\int_{0}^{1}\frac{1+we^{-2\pi i x}}{1-we^{-2\pi i x}}d\mu.$$
\end{thm*}

We will also make use of Lemma 2 in \cite{DeBranges1960Some} that has a similar flavor to the Herglotz Theorem:

\begin{lem*}[De Branges] \label{L:deBranges2}
Let $U$ be a unitary operator on Hilbert space $H$ and $C$ be a closed subspace of $H$. For each $w\in \mathbb{D}$, there is a unique operator $B(w)$ on $C$ such that
\begin{equation}{\label{DL}}\langle \frac{1+wU^{*}}{1-wU^{*}}a,c\rangle=\langle \frac{1+B(w)}{1-B(w)}a,c\rangle\end{equation} for any $a,c\in C$. Furthermore, $B(w)$ is an analytic operator on $\mathbb{D}$ and $||B(w)||\leq |w|$ for all $w\in \mathbb{D}$.
\end{lem*}

By $B(w)$ being an analytic operator, we mean for each $w\in \mathbb{D}$,
$$B(w)=\sum_{n=0}^{\infty}B_{n}w^{n}\in B(C)$$ where for each $n$, $B_{n}\in B(C)$, and the convergence is in operator norm.
Note that we replace $wB(w)$ in \cite{DeBranges1960Some} with $B(w)$.

\subsection{Direct integrals}

Furthermore, we will use the definition of direct integral decomposition in a Hilbert space from \cite{Garrett1996Good} see also \cite{naimark1974direct}\cite{nielsen1980direct}.

An example of a direct integral we will consider later is the following: Let $\mu$ be a Borel probability measure on $[0,1)^{d}$ for some $d\in \mathbb{N}$. Consider the following Hilbert space of power series: $$L^{2}(\mu)(w):=\{\sum_{n=0}^{\infty}f_{n}(x)w^{n}: f_{n}(x)\in L^{2}(\mu), \sum_{n=0}^{\infty}||f_{n}||^{2}<\infty\}$$ where
$$\langle \sum_{n=0}^{\infty}f_{n}(x)w^{n}, \sum_{n=0}^{\infty}g_{n}(x)w^{n}\rangle :=
\sum_{n=0}^{\infty}\langle f_{n},g_{n}\rangle.$$
In general for any Hilbert space $H$, $H(w)$ will denote the Hilbert space of power series defined similarly as above.
One can check that
$$L^{2}(\mu)(w)=\int_{[0,1)^{d}}^{\oplus}H^{2}(\mathbb{D})d\mu$$
where 
$$\{\sum_{n=0}^{\infty}f_{n}(x)w^{n}: f_{n}(x) \text{ are measurable for each $n$ and} \ \sum_{n=0}^{\infty}f_{n}(x)w^{n}\in H^{2}(\mathbb{D}) \ for \ \mu \ a.e. \ every \ x \}$$ is our choice of measurable sections, and the inner products of these spaces coincide.

Therefore, for $F(w)=\sum_{n=0}^{\infty}f_{n}(x)w^{n}\in L^{2}(\mu)(w)$, we are justified in replacing $F(w)$ with $\int_{[0,1)^{d}}^{\oplus}F^{x}d\mu$ where for $\mu$ a.e. fixed $x$, $F^{x}(w)=\sum_{n=0}^{\infty}f_{n}(x)w^{n}\in H^{2}(\mathbb{D})$.

\section{Proof of Theorem A} \label{sec:d-dim}

Our goal in this section is to prove Theorem A. Furthermore,
we will show that the coefficients of this Fourier expansion are obtained from an extended Kaczmarz algorithm \cite{Kacz37}. 

We will now assume for the rest of the section that we have fixed a $x_{d}$ slice singular Borel probability measure $\mu$ on $[0,1)^{d-1}\times [0,1)$ where $d\geq 2$ and the slice measures are $\{\gamma^{x_{1},\dots ,x_{d-1}}\}$ and the marginal measure on $[0,1)^{d-1}$ is $\mu_{1}$.

An outline of the proof of this section is as follows: In the first subsection, we define a sequence of operators with a goal of showing their effectiveness, analogously with the Kaczmarz algorithm. The next subsection we prove an effectiveness theorem from the theory of the operator Kaczmarz algorithm using a Herglotz Theorem type argument. Finally, in the last subsection we calculate the coefficients of the expansion using recursive formulas and induction.

\subsection{Operator Kaczmarz algorithm }

Natterer \cite{Natterer1986Mathematics} introduced the Kaczmarz algorithm for bounded operators, and the following discussion is inspired from \cite{Herr2022Fourier}. We construct a sequence of operators and use the theory of the Kaczmarz algorithm to get effective convergence.

\begin{defn}
Define a sequence of operators:
$$R(n,d):L^{2}(\mu)\to L^{2}(\mu_{1})$$ for $n\in \mathbb{N}$ where
$$R(n,d)(f(x_{1},\dots, x_{d}))=\int_{0}^{1}f(x_{1},\dots, x_{d})e^{-2\pi inx_{d}}d\gamma^{x_{1},\dots, x_{d-1}}(x_{d}),$$ and the unitary
$$S(d):L^{2}(\mu)\to L^{2}(\mu)$$ where
$$S(d)(f(x_{1},\dots, x_{d}))=e^{2\pi i x_{d}}f(x_{1},\dots, x_{d}).$$
\end{defn}
It is easy to see that for each $n$,
$$R(n,d)^{*}(h(x_{1},\dots ,x_{d-1}))=e^{2\pi i nx_{d}}h(x_{1},\dots, x_{d-1})$$
so that $R(n,d)^{*}$ is a isometry, and
$$R(n+j,d)=R(n,d)S^{-j}(d)$$
for any $j\geq 0$.

Our goal is to show that the $\{R(n,d)\}$ are effective operators via the operator Kaczmarz algorithm defined in \cite{Herr2022Fourier}: 
\begin{defn}
Given Hilbert spaces $H$ and $K$ with a sequence of operators $\{R_{n}\}_{n=0}^{\infty}$ where for each $n$, $R_{n}\in B(H,K)$, define $\{G_{n}\}\in B(H, K)$ recursively as follows:
$$G_{0}=R_{0}$$ and for $n\geq 1$
$$G_{n}=R_{n}-\sum_{k=0}^{n-1}R_{n}R_{k}^{*}G_{k}.$$
We say that the $\{R_{n}\}$ are \textbf{effective} if for all $f\in H$,
$$\sum_{n=0}^{\infty}R_{n}^{*}G_{n}f=f$$ where the convergence is in $H$.
\end{defn}
We will achieve our goal by verifying the hypothesis of the following theorem in \cite{Herr2022Fourier}:

\begin{thmalpha}[Herr, Jorgensen, and Weber]\label{effectivethm}
Given Hilbert spaces $H$ and $K$ and a sequence of operators $\{R_{n}\}\subseteq B(H,K)$, let $M$ be a matrix with entries indexed by $\mathbb{N}^{2}$ that is strictly lower triangular where for $j>k$, the $jk$th entry is $R_{j}R_{k}^{*}$.
Then for $I$ acting as the identity matrix on $\oplus_{n=0}^{\infty}K$, there exists a matrix $U$ such that
$$(I+M)(I+U)=I.$$

Suppose that every $R_{n}^{*}$ is an isometry. Then $\{R_{n}\}$ are effective if and only if $\bigcup[ker(R_{n})]^{\perp}$ is linearly dense in $H$ and $U$
    is a partial isometry on $\oplus_{n=0}^{\infty}K$.
\end{thmalpha}

\subsection{Main proof of Theorem A}
To prove Theorem A, we start by verifying the hypothesis of Theorem B. 
Recall that $\mu$ is assumed $x_{d}$ slice singular with marginal measure $\mu_{1}$ on $[0,1)^{d-1}$ that is assumed $x_{d-1}$ slice singular, and $\mu$ has a family of slice measures $\{\gamma^{x_{1},\dots,x_{d-1}}\}$ on $[0,1)$ that are singular for $\mu_{1}$ almost every $(x_{1},\dots, x_{d-1})$. 

Suppose that there is an $f\in L^{2}(\mu)$ that is in $ker(R(n,d))$ for all $n$. Then we have for $\mu_{1}$ almost every $(x_{1},\dots, x_{d-1})$,
$$\int_{0}^{1}f(x_{1},\dots, x_{d})e^{-2\pi in x_{d}}d\gamma^{x_{1},\dots, x_{d-1}}(x_{d})=0 $$
for all $n$.
By the density of $\{e^{2\pi i nx_{d}}\}_{n=0}^{\infty}$ in the one dimensional case and the Rokhlin disintegration, for $\mu_{1}$ almost every $(x_{1},\dots, x_{d-1})$,
$f(x_{1},\dots, x_{d})=0$ in $L^{2}(\gamma^{x_{1},\dots, x_{d-1}})$, so $f=0$ in $L^{2}(\mu).$

Now for the last condition of Theorem B, by the Herglotz Theorem, for almost every $(x_{1},\dots, x_{d-1})$, there is an inner function $b(w)^{x_{1},\dots, x_{d-1}}$ corresponding to $\gamma^{x_{1},\dots, x_{d-1}}$.
Furthermore, by the de Branges Lemma, there is an analytic operator $B(w)$ acting on $L^{2}(\mu_{1})$ for each $w\in \mathbb{D}$ such that:

\begin{equation}\label{FC}
\begin{split}
\ip{ \frac{1+B(w)}{1-B(w)}a,c} & =\ip{ \frac{1+wS^{*}(d)}{1-wS^{*}(d)}a,c} \\
& =\int_{[0,1)^{d-1}}\int_{[0,1)}\frac{1+we^{-2\pi i x_{d}}}{1-we^{-2\pi i x_{d}}}a(x_{1},\dots, x_{d-1})\overline{c(x_{1},\dots, x_{d-1})}d\gamma^{x_{1},\dots,x_{d-1}}d\mu_{1} \\
& =\int_{[0,1)^{d-1}}a(x_{1},\dots, x_{d-1})\overline{c(x_{1},\dots, x_{d-1})}
\left[\frac{1+b(w)^{x_{1},\dots, x_{d-1}}}{1-b(w)^{x_{1},\dots, x_{d-1}}}\right]d\mu_{1} \\
\end{split}
\end{equation}
for any $a,c\in L^{2}(\mu_{1})$. Now for each $w$, define 
$$S^{+}(w)=\sum_{n=0}^{\infty}R(0,d)S^{-n}(d)R(0,d)^{*}w^{n}.$$
Notice that for each $a,c\in L^{2}(\mu_{1})$,
$$\langle (2S^{+}(w)-I)a,c\rangle=\ip{ \frac{1+wS^{*}(d)}{1-wS^{*}(d)}a,c}.$$
This gives us that $$B(w)=1-[S^{+}(w)]^{-1}.$$
Furthermore, $S^{+}(w)$ acts on polynomials in $L^{2}(\mu_{1})(w)$ by matrix multiplication in the form of $I+M$ as defined in Theorem B so that $B(w)$ acts on polynomials in $L^{2}(\mu_{1})(w)$ by matrix multiplication in the form of $-U$. Therefore, we will show that $B(w)$ acting on $L^{2}(\mu_{1})(w)$ as power series multiplication
is an isometry, and this will verify the hypothesis of Theorem B.

By the functional calculus and line (\ref{FC}), for any fixed $w$,
$B(w)$ acts on $L^{2}(\mu_{1})$ by multiplication by $b^{x_{1},\dots ,x_{d-1}}(w)$.

By the direct integral discussion in section two, consider
$$F(w)=\sum_{n=0}^{\infty}f_{n}(x_{1},\dots, x_{d-1})w^{n}=\int_{[0,1)^{d-1}}^{\oplus}F^{x_{1},\dots,x_{d-1}}(w)d\mu_{1}\in \int_{[0,1)^{d-1}}^{\oplus}H^{2}(\mathbb{D})d\mu_{1}=L^{2}(\mu_{1})(w)$$
We have
$$||B(w)F(w)||_{L^{2}(\mu_{1})(w)}^{2}=||b^{x_{1},\dots ,x_{d-1}}(w)\sum_{n=0}^{\infty}f_{n}(x_{1},\dots, x_{d-1})w^{n}||_{L^{2}(\mu_{1})(w)}^{2}$$
$$=\int_{[0,1)^{d-1}}||b^{x_{1},\dots x_{d-1}}(w)F^{x_{1},\dots,x_{d-1}}(w)||^{2}_{H^{2}(\mathbb{D})}d\mu_{1}.$$
Because multiplication by $b^{x_{1},\dots x_{d-1}}(w)$ on $H^{2}(\mathbb{D})$ is an isometry for $\mu_{1}$ almost every $(x_{1},\dots, x_{d-1})$, $B(w)$ acts as an isometry on $L^{2}(\mu_{1})(w)$. Therefore, we have shown that $\{R(n,d)\}$ are effective.

Now we have for all $f\in L^{2}(\mu)$,
$$f=\sum_{n=0}^{\infty}[G(n,d)f]e^{2\pi i nx_{d}}.$$ 

Now for each $n$, $G(n,d)f\in L^{2}(\mu_{1})$, and since $\mu_{1}$ is slice singular in $d-1$ dimensions, we may apply a recursive argument to get
$$f=\sum_{n=0}^{\infty}[\sum_{m=0}^{\infty}[G(m,d-1)G(n,d)f]e^{2\pi i mx_{d-1}}]e^{2\pi i nx_{d}}.$$

With the one dimensional case coming from \cite{Herr2017Fourier}, we get that there a bounded linear functional $G$ such that
$$f=\sum_{n_{d}=0}^{\infty}\dots \sum_{n_{1}=0}^{\infty}G(f)e^{2\pi i (n_{1}x_{1}+\dots +n_{d}x_{d})}.$$

\subsection{Coefficients} \label{ssec:coeff}

To calculate the coefficients from Theorem A, we show the following:
\begin{prop}\label{auxformula}
For any $n,d$, and $f\in L^{2}(\mu)$,
$$G(n,d)f=\int_{[0,1)}f(x_{1},\dots, x_{d})\overline{g_{n}^{x_{1},\dots, x_{d-1}}}(x_{d})d\gamma^{x_{1},\dots ,x_{d-1}}$$
where $\{g_{n}^{x_{1},\dots, x_{d-1}}(x_{d})\}$ is the auxiliary sequence associated with $\gamma^{x_{1},\dots ,x_{d-1}}$.
\end{prop}
\begin{proof}
We have for any $n$ and any $d$, from the recursive definition of $G(n,d)$,
\begin{equation}\label{C}R(n,d)=\sum_{j=0}^{n}R(n,d)R(j,d)^{*}G(j,d).\end{equation}
We show that $\int_{[0,1)}f(x_{1},\dots, x_{d})\overline{g_{n}^{x_{1},\dots, x_{d-1}}}(x_{d})d\gamma^{x_{1},\dots ,x_{d-1}}$ satisfies line (\ref{C}). An easy calculation gets that for any $h(x_{1},\dots, x_{d-1})\in L^{2}(\mu_{1})$,
$$R(n,d)R(j,d)^{*}h(x_{1},\dots, x_{d-1})=\widehat{\gamma^{x_{1},\dots, x_{d-1}}}(n-j)h(x_{1},\dots, x_{d-1}).$$
Also, it is easy to see from the recursive definition of the auxiliary sequence, we know for any fixed $(x_{1},\dots, x_{d-1})$,
$$\sum_{j=0}^{n}\overline{\widehat{\gamma^{x_{1},\dots, x_{d-1}}}(n-j)}g_{j}^{x_{1},\dots, x_{d-1}}(x_{d})=e^{2\pi i n x_{d}}.$$
Now note that $\int_{[0,1)}f(x_{1},\dots, x_{d})\overline{g_{n}^{x_{1},\dots, x_{d-1}}}(x_{d})d\gamma^{x_{1},\dots ,x_{d-1}}$ satisfies the desired equation.
\end{proof}
To close this section, recall that for $\mu$ that is $x_{d}$ slice singular, $\{g_{n_{k}}^{x_{1},\dots, x_{k-1}}\}$ is the auxiliary sequence from the slice measures of $\mu_{d-k}$ for each $k$, and $\{g_{n_{1}}\}$ is the auxiliary sequence of $\mu_{d-1}$, which are defined in Equation (\ref{Eq:aux}).
\begin{prop}
Under the hypothesis from Theorem A, $$c_{n_{1}\dots n_{d}}=\langle f(x_{1},\dots, x_{d}),g_{n_{d}}^{x_{1},\dots, x_{d-1}}\dots g_{n_{1}}\rangle_{\mu}$$
for all $n_{1},\dots, n_{d}\in \mathbb{N}$.
Furthermore, $\{g_{n_{d}}^{x_{1},\dots, x_{d-1}}\dots g_{n_{1}}\}$ is a Parseval frame in $L^{2}(\mu)$.
\end{prop}
\begin{proof}
We proceed with an induction on $d$.
Recall that for all $f\in L^{2}(\mu),$
$$f=\sum_{n_{d}=0}^{\infty}[G(n_{d},d)f]e^{2\pi i n_{d}x_{d}}.$$
For the $d=2$ case, we have by Proposition \ref{auxformula},
$$G(n_{2},2)f=\int_{[0,1)}f(x_{1},x_{2})\overline{g_{n_{2}}^{x_{1}}(x_{2})}d\gamma^{x_{1}}$$ for all $n_{2}$.
Therefore, by the one dimensional case \cite{Herr2017Fourier} for all $n_{2}$,
\begin{equation}
\begin{split}
G(n_{2},2)f & =\sum_{n_{1}=0}^{\infty}\langle \int_{[0,1)}f(x_{1},x_{2})\overline{g_{n_{2}}^{x_{1}}(x_{2})}d\gamma^{x_{1}},g_{n_{1}}(x_{1})\rangle_{\mu_{1 }}e^{2\pi i n_{1} x_{1}}\\
& =\sum_{n_{1}=0}^{\infty}\langle f(x_{1},x_{2}),g_{n_{2}}^{x_{1}}g_{n_{1}}(x_{1})\rangle_{\mu}e^{2\pi i n_{1}x_{1}}.
\end{split}
\end{equation}
It follows that
$$c_{n_{1}n_{2}}=\langle f(x_{1},x_{2}),g_{n_{2}}^{x_{1}}g_{n_{1}}(x_{1})\rangle_{\mu}$$ for all $n_{1},n_{2}\in \mathbb{N}$.

Now let $f\in L^{2}(\mu)$. Recall that $\{g_{n_{1}}(x_{1})\}$ is a Parseval frame in $L^{2}(\mu_{1})$ and that for $\mu_{1}$ almost every $x_{1},$ $\{g_{n_{2}}^{x_{1}}\}$ is a Parseval frame in $L^{2}(\gamma^{x_{1}})$. Consider the following:
\begin{equation}\label{pf}
\begin{split}
\sum_{n_{2}}\sum_{n_{1}}|\langle f, g_{n_{2}}^{x_{1}}g_{n_{1}}(x_{1})\rangle_{\mu}|^{2} & =\sum_{n_{2}}\sum_{n_{1}}|\langle \langle f,g_{n_{2}}^{x_{1}}\rangle_{\gamma^{x_{1}}},g_{n_{1}}\rangle_{\mu_{1}}|^{2}\\
& =\sum_{n_{2}}||\langle f,g_{n_{2}}^{x_{1}}\rangle_{\gamma^{x_{1}}}||_{\mu_{1}}^{2}\\
& =\int_{0}^{1}\sum_{n_{2}}|\langle f,g_{n_{2}}^{x_{1}}\rangle_{\gamma^{x_{1}}}|^{2}d\mu_{1}\\
& =\int_{0}^{1}||f||_{\gamma^{x_{1}}}^{2}\mu_{1}\\
& =||f||_{\mu}^{2}.
\end{split}
\end{equation}
Therefore, $\{g_{n_{2}}^{x_{1}}g_{n_{1}}\}$ is a Parseval frame in $L^{2}(\mu)$.

For the induction step, we know that by Proposition \ref{auxformula} for all $n_{d}$,
$$G(n_{d},d)f=\langle f, g_{n_{d}}^{x_{1},\dots, x_{d-1}}\rangle_{\gamma^{x_{1},\dots, x_{d-1}}}\in L^{2}(\mu_{1}).$$
By our inductive hypothesis, $\langle f, g_{n_{d}}^{x_{1},\dots, x_{d-1}}\rangle_{\gamma^{x_{1},\dots, x_{d-1}}}$ has $\{\langle \langle f, g_{n_{d}}^{x_{1},\dots, x_{d-1}}\rangle_{\gamma^{x_{1},\dots, x_{d-1}}}, g_{n_{d-1}}^{x_{1},\dots, x_{d-2}}\dots g_{n_{1}}\rangle_{\mu_{1}}\}$
as its Fourier coefficients. The first result now follows.

Also, showing $\{g_{n_{d}}^{x_{1},\dots, x_{d-1}}\dots g_{n_{1}}\}$ is a Parseval frame in $L^{2}(\mu)$ is similar to the base case since we assume that $\{g_{n_{d-1}}^{x_{1},\dots, x_{d-2}}\dots g_{n_{1}}\}$ is a Parseval frame in $L^{2}(\mu_{1})$, and we know that for $\mu_{1}$ almost every $(x_{1},\dots x_{d-1})$, $\{g_{n_{d}}^{x_{1},\dots, x_{d-1}}\}$ is a Parseval frame in $L^{2}(\gamma^{x_{1},\dots, x_{d-1}})$.
\end{proof}

\subsection{Examples of slice singular measures}
To illustrate commonality of these slice singular measures, we describe a couple of examples.
\begin{ex}
Suppose $d\geq 2$, and let $\mu=\prod_{k=1}^{d}\mu_{k}$ be a product of $d$ finite singular Borel measures on $[0,1)$. For any $f\in L^{1}(\mu)$ that is positive on $[0,1)^{d}$ where
$\int_{[0,1)^{d}}f d\mu=1$, we have $f d\mu$ is $x_{d}$ slice singular. To see this, one can see that the slice measures of $f d\mu$ are $$\gamma^{x_{1},\dots,x_{d-1}}=\frac{f(x_{1},\dots,x_{d-1},x_{d})\mu_{d}}{\int_{0}^{1}f(x_{1},\dots,x_{d-1},x_{d})d\mu_{d}(x_{d})},$$ which exists for $\prod_{k=1}^{d-1}\mu_{k}$ almost every $(x_{1},\dots, x_{d-1})$ by Fubini-Tonelli Theorem. Also, the marginal measure of $f d\mu$ is $$\tilde{\mu_{1}}=\big[\int_{0}^{1}f(x_{1},\dots,x_{d-1},x_{d})d\mu_{d}(x_{d}) \big] \big[\prod_{k=1}^{d-1}\mu_{k} \big].$$ Then for $\tilde{\mu_{1}}$ almost every $(x_{1},\dots, x_{d-1})$, $\gamma^{x_{1},\dots,x_{d-1}}$ are absolutely continuous to $\mu_{d}$ and are therefore singular, and $\tilde{\mu_{1}}$ is slice singular by a recursive argument.

Furthermore, if we consider one dimensional slice measures from $\mu$ by Rokhlin's theorem that come from a variable other than $x_{d}$, that family of slice measures would still be singular for almost every member of that family with respect to the marginal measure. Additionally, that marginal measure would have singular slice measures on any of its variables as well. 
\end{ex}
This bring us to the following definition:
\begin{defn}
For a Borel probability measure $\mu$ on $[0,1)^{d}$ where $d\geq 2$, we define $\mu$ as \textbf{slice singular in any variable order} if for any $k$, the one dimensional family of slice measures on variable $x_{k}$ are singular almost everywhere with respect to the marginal measure $\tilde{\mu_{k}}$ on variables $(x_{1},\dots, x_{k-1},x_{k+1},\dots x_{d})$, and $\tilde{\mu_{k}}$ is slice singular in any variable order.

If $d=2$, then each $\tilde{\mu_{k}}$ being slice singular in any variable order means $\tilde{\mu_{k}}$ are singular.
\end{defn}
More examples of slice singular measures arise from iterated functions systems (IFS), which are discussed in \cite{Herr2022Fourier}. As a reminder, the following Hutchinson's Theorem \cite{Hutchinson1981Fractals} establishes that IFS of contractions give rise to Borel probability measures.

\begin{thm*}[Hutchinson]
Let $(X,d)$ be a complete metric space and $\phi=\{\phi_{1},\dots ,\phi_{N}\}$ be a finite set of contractions on $X$. Then there exists a unique compact set $A$ such that
$$A=\bigcup_{k=1}^{N}\phi_{k}(A).$$
Furthermore, $A$ is the closure of the set of fixed points of finite compositions of members of $\phi$.

Moreover, given $w_{1},\dots, w_{N}>0$ such that $\sum w_{k}=1$, there exists a unique Borel probability measure $\mu$ called the invariant measure that is supported on $A$ such that
$$\int fd\mu=\sum_{k=1}^{N}w_{k}\int f\circ \phi_{k} d\mu$$
holds for every continuous function $f$.
\end{thm*}

To understand which IFS generate slice singular measures, we establish some Lemmas:

\begin{lem}\label{onedimmarginal}
    If every marginal measure in dimension $d-1$ is slice singular in any variable order of a Borel probability measure $\mu$ on $[0,1)^{d}$ where $d\geq 3$, then for any $1\leq k\leq d$ there exists $A_{k}\subseteq [0,1)$ with Lebesgue measure one such that
$$\mu([0,1)\times \dots \times A_{k} \times \dots \times [0,1))=0 $$
where $A_{k}$ is in the $x_{k}$ variable spot.
\end{lem}
\begin{proof}
Assume that $d=3$. Since each marginal measure in two variables is slice singular in any variable order, both of the one dimensional marginals measures of each
marginal measure in two variables are singular.
Therefore, there exists Borel $B_{1},B_{2}\subseteq [0,1)$ with Lebesgue measure one for each marginal measure in two variables such that $[0,1)\times B_{1}$ and $B_{2}\times [0,1)$ are measure zero with respect to each marginal measure in two variables respectively. The claim follows for $d=3$. The case for $d>3$ follows similarly by induction. 
\end{proof}

\begin{lem}\label{d-1dimmarginal}
A Borel probability measure $\mu$ on $[0,1)^{d}$ where $d\geq 2$ is slice singular in any variable order if and only if every marginal measure in dimension $d-1$ is slice singular in any variable order.
\end{lem}
\begin{proof}
The direction of assuming that $\mu$ is slice singular in any variable order follows by definition.

Now suppose  every marginal measure in dimension $d-1$ is slice singular in any variable order.
Suppose for the sake of contradiction that there is for example a Borel set $B\subseteq [0,1)^{d-1}$ of $\mu_{1}$ positive measure where for $(x_{1},\dots ,x_{d-1})\in B$, $\gamma^{x_{1},\dots ,x_{d-1}}$ are not singular, meaning that $\mu$ is not $x_{d}$ slice singular. From Lemma \ref{onedimmarginal}, find $A_{d}\subseteq [0,1)$ of Lebesgue measure one such that 
$$\mu([0,1)\times \dots \times A_{d})=0.$$
But notice that $$\gamma^{x_{1},\dots ,x_{d-1}}(A_{d})>0$$ for $(x_{1},\dots ,x_{d-1})\in B$ so we also have
$$\mu([0,1)\times \dots \times A_{d})\geq \int_{B}\gamma^{x_{1},\dots ,x_{d-1}}(A_{d})d\mu_{1}>0,$$ which is a contradiction. Showing that $\mu$ is slice singular in other directions is similar.
\end{proof}

These Lemmas give us a nice criteria to check when measures are slice singular in any variable order via the following Theorem.

\begin{defn}
Let $\pi_{k}: [0,1)^{d}\to [0,1)$ be the projection onto the $k$th coordinate.
\end{defn}
\begin{thm}\label{identityslicesingular}
A Borel probability measure $\mu$ on $[0,1)^{d}$ is slice singular in any variable order if and only if for all $1\leq k\leq d$, $\mu \circ \pi_{k}^{-1}$, the one-dimensional marginal measure on variable $x_{k}$, is a singular measure on $[0,1)$.
\end{thm}
\begin{proof}
The direction assuming that $\mu$ is slice singular in any variable order is proven by Lemma \ref{onedimmarginal}.

Now assume that for all $1\leq k\leq d$, $\mu \circ \pi_{k}^{-1}$ is a singular measure on $[0,1)$. We proceed by induction on $d$.

The case when $d=2$ follows from the proof of Lemma \ref{d-1dimmarginal}.

For the case for $d>2$, suppose the claim follows for $d-1$ dimensions. We will show that every marginal measure in $d-1$ dimensions is slice singular in any variable order so that the result will follow from Lemma \ref{d-1dimmarginal}. Take any marginal measure of $\mu$ on $d-1$ dimensions, call it $v$ and let $1\leq k\leq d-1$. Consider 
$$v\circ \pi_{k}^{-1}=\mu \circ \tilde{\pi}_{j}^{-1}\circ \pi_{k}^{-1}$$
where $\tilde{\pi}_{j}: [0,1)^{d}\to [0,1)^{d-1}$ and $\tilde{\pi}_{j}$ is the projection that excludes the coordinate $j$. Then $v\circ \pi_{k}^{-1}=\mu \circ \pi_{m}^{-1}$ where $m=k+1$ if $j\leq k$ and $m=k$ if $j>k$, which is a singular measure by assumption. Then by induction assumption $v$ is slice singular in any variable order. Since $v$ was an arbitrary marginal measure of $\mu$ on $d-1$ dimensions, the result follows.
\end{proof}

Now to see how to use this criteria for measures generated by an IFS, we establish some notation and a Lemma that follows easily from the Hutchinson theorem. For details see \cite[Lemma 1]{Herr2022Fourier}.

\begin{defn}
For a map $\phi: \mathbb{R}^{d}\to \mathbb{R}^{d}$, we say $\phi$ is \textbf{Cartesian} if there exists mappings $\phi_{1},\phi_{2},\dots, \phi_{d}: \mathbb{R}\to \mathbb{R}$ such that
$$\phi(x_{1},\dots,x_{d})=(\phi_{1}(x_{1}),\dots, \phi_{d}(x_{d})).$$

For a finite set of Cartesian maps on $[0,1)^{d}$, $\phi=\{\phi_{1},\dots, \phi_{N}\}$, define a $d\times N$ matrix of functions by $$\tilde{\phi}=[\phi_{kj}]$$ where $\phi_{kj}\circ \pi_{k}=\pi_{k}\circ \phi_{j} $.
\end{defn}

\begin{lem*}[Herr, Jorgensen, and Weber]
Let $\phi=\{\phi_{1},\dots, \phi_{N}\}$ be a finite set of Cartesian maps that are strict contractions on $[0,1)^{d}$ with weights $\{w_{1},\dots ,w_{N}\}$ and invariant measure $\mu$. Then for any $k$, the $kth$ row of $\tilde{\phi}$
is a multi-set of contractions on $[0,1)$ with weights $\{w_{1},\dots ,w_{N}\}$, and its invariant measure is $\mu \circ \pi_{k}^{-1}$.
\end{lem*}
We also use the following result of Kakutani \cite{kakutani1948equivalence} for understanding when IFS generate measures that are singular.
\begin{thm*}[Kakutani]
Suppose that $\{\phi_{1},\dots ,\phi_{N}\}$ are contractions acting on the metric space $(X,d)$, and let $\mu$ and $\tilde{\mu}$ be the invariant measures arising from these contractions with weights $\{w_{1},\dots ,w_{N}\}$ and $\{\tilde{w}_{1},\dots ,\tilde{w}_{N}\}$ respectively. Then exactly one of the following holds:
\begin{enumerate}
    \item $\mu=\tilde{\mu}$, which is equivalent to $w_{j}=\tilde{w_{j}}$ for all $j$.
    \item $\mu$ and $\tilde{\mu}$ are mutually singular.
\end{enumerate}
\end{thm*}
This gives us our criteria for testing the measure corresponding to an IFS for being slice singular in any variable order, which follows from Theorem \ref{identityslicesingular}.

\begin{prop}\label{identityifs}
Let $\phi=\{\phi_{1},\dots, \phi_{N}\}$ be a finite set of Cartesian maps that are strict contractions on $[0,1)^{d}$ with weights $\{w_{1},\dots ,w_{N}\}$ and invariant measure $\mu$. Then $\mu$ is slice singular in any variable order if and only if for all $k$, Lebesgue measure is not the invariant measure of the contractions given by the $k$th row of $\tilde{\phi}$ with weights $\{w_{1},\dots ,w_{N}\}$.
\end{prop}

\begin{rem}
If particular, suppose that for each $k$, the $k$th row of contractions in $\tilde{\phi}$ are reduced to a set of contractions $\{\psi_{k1},\dots ,\psi_{kN'}\}$ by eliminating repetition and the weights $\{w_{k},\dots ,w_{N}\}$ are converted to a weight set $\{w'_{k1},\dots ,w'_{kN'}\}$ by summing the weights of $w_{i}$ corresponding to repeated contractions. Suppose further that the images of $\{\psi_{k1},\dots ,\psi_{kN'}\}$ partition $[0,1)$. Then Lebesgue measure can be the invariant measure of the system with contractions $\{\psi_{k1},\dots ,\psi_{kN'}\}$ with the appropriate weights for all $k$, and if for all $k$, the weights from the system where Lebesgue measure is the invariant measure with contractions $\{\psi_{k1},\dots ,\psi_{kN'}\}$ do not correspond to $\{w'_{k1},\dots ,w'_{kN'}\}$, then $\mu$ is slice singular in any variable order.
\end{rem}

\begin{ex}[Menger Sponge]

Now for a concrete example, for any $1\leq p\leq 3$, let
$\lambda_{0}(x_{p})=\frac{x_{p}}{3}$, $\lambda_{1}(x_{p})=\frac{x_{p}+1}{3}$, and $\lambda_{2}(x_{p})=\frac{x_{p}+2}{3}$.
Consider the invariant measure of the following finite sequence of Cartesian maps that are strict contractions on $[0,1)^{3}$, $$\phi=\{\phi_{j}\}_{j=1}^{20}$$ where the weights are all equal
and the columns of the matrix $\tilde{\phi}$ are $[\lambda_{l}(x_{1}) \ \lambda_{m}(x_{2}) \ \lambda_{n}(x_{3}))]^{T}$ where $0\leq l,m,n\leq 2$, and at most one entry of  $(l,m,n)$ is one.

This generates the three dimensional fractal known as the Menger sponge. Intuitively, the system starts with a $3\times 3\times 3$ cube and removes the center cube from each face as well as the most center cube. This process repeats on each of the remaining cubes. 

Note that under appropriate permutations, all the $kth$ rows of $\tilde{\phi}$ are
\[ \phi_{kj}=\begin{cases} 
      \lambda_{1}(x_{k}) \ j=1,\dots,8 \\
      \lambda_{2}(x_{k}) \ j=9,\dots,12 \\
      \lambda_{3}(x_{k}) \ j=13,\dots,20
   \end{cases}
.\]
Note that Lebesgue measure on $[0,1)$ is the invariant measure of the system with those twenty contractions above provided the weights correspond to varying repetitions of the $\lambda_{j}$. Therefore, by the Kakutani Dichotomy theorem, the marginal measures of the invariant measure $\mu$ for $\phi=\{\phi_{j}\}_{j=1}^{20}$ with equal weights are singular with respect to Lebesgue measure and hence $\mu$ is slice singular in any variable order by Proposition \ref{identityifs}.
\end{ex}

\section{A Normalized Cauchy Transform in higher dimensions} \label{sec:NCT}
Our goal is this section is to define a Normalized Cauchy Transform in two dimensions and discuss some of its properties. In particular, we provide a description of the image of the Normalized Cauchy Transform that we define using our coefficients from Theorem A. We will also discuss a Normalized Cauchy Transform in higher dimensions at the end of this section. 

\subsection{One Variable Normalized Cauchy Transform}
Recall the definition of the Normalized Cauchy Transform on $\mathbb{T}$.

\begin{defn}
Let $\mu$ be a finite Borel measure on $[0,1)$. Recall that the \textbf{Cauchy transform}, $C_{\mu}$, is a map from $L^{1}(\mu)$ to analytic functions on the disk such that
$$[C_{\mu}(f)](w)=\int_{0}^{1}\frac{f(x)d\mu}{1-we^{-2\pi i x}}.$$
Now specifically for $\mu$ that is a singular Borel probability measure on $[0,1)$, let $b(w)$ be the corresponding inner function of $\mu$ via the Herglotz Theorem, and define $\mathcal{H}(b)=[b(w)H^{2}(\mathbb{D})]^{\perp}$.
Define $V_{\mu}: L^{2}(\mu)\to \mathcal{H}(b)$
by \begin{equation}V_{\mu}(f)=\frac{C_{\mu}(f)}{C_{\mu}(1)} = \int_{0}^{1} \dfrac{f(x) d\mu}{(1 - (\cdot)e^{-2\pi i x}) [C_{\mu}(1)](\cdot)},\end{equation}
which is a unitary operator called the \textbf{Normalized Cauchy Transform} \cite{Clark72,Aleks89a,Sarason1994Sub-Hardy}.
\end{defn}

Before we define a Normalized Cauchy Transform in higher dimensions, we draw on some results in \cite{DeBranges1960Some} and \cite{Herr2017Fourier}.
\begin{lem*}[Herr and Weber]
Let $\mu$ be a singular Borel probability measure on $[0,1)$ and $b(w)$ be the corresponding inner function via the Herglotz theorem. Then
\begin{equation}{\label{inner}}
b(w)=1-\frac{1}{[C_{\mu}(1)](w)}.
\end{equation}
\end{lem*}

\begin{thm*}[Herr and Weber]
Let $\mu$ be a singular Borel probability measure on $[0,1)$ with auxiliary sequence $\{g_{n}\}_{n=0}^{\infty}$ generated from $\{e^{2\pi i nx}\}_{n=0}^{\infty}$ in $L^{2}(\mu)$.
Then \begin{equation}{\label{NCT}}[V_{\mu}(f)](w)=\sum_{n=0}^{\infty}\langle f,g_{n}\rangle w^{n}.\end{equation}
\end{thm*}
Therefore, the Normalized Cauchy Transform may be defined in terms of the auxiliary sequence $\{g_{n}\}$. The following is an easy consequence of equation (\ref{inner}).
\begin{cor}
For $\mu$ and $b(w)$ defined in the last Theorem,
\begin{equation}{\label {ifandnct}}V_{\mu}(e^{-2\pi i x})=\frac{b(w)}{w}.\end{equation}
Furthermore,
for any $f\in L^{2}(\mu)$,
\begin{equation}{\label{backward}}\frac{V_{\mu}(f)-V_{\mu}(f)(0)}{w}=V_{\mu}(e^{-2\pi i x}(f-\langle f,1\rangle)).\end{equation}
\end{cor}

The following is compiled from Lemma 4,11, and 12 in \cite{DeBranges1960Some}, and it ties model spaces with the vector version of the Herglotz theorem:
\begin{thm*}[De-Branges]
Let $C$ be a Hilbert space, let $B(z)$ be an analytic operator acting as an isometry on $C(z)$ by power series multiplication such that $B(0)=0$, and let 
$$\mathcal{H}(B)=[B(z)C(z)]^{\perp}.$$

Then $U:\mathcal{H}(B)\to \mathcal{H}(B)$ where
$$U^{*}(f(z))=\frac{f(z)-f(0)}{z}+\frac{B(z)}{z}f(0)$$
is a unitary such that the equation (\ref{DL}) holds. Also,
$$\mathcal{H}(B)=\overline{span\{U^{-n}(c): c\in C, n\in \mathbb{N}\}}.$$
\end{thm*}

\subsection{A Two Variable Normalized Cauchy Transform}
Let $\mu$ be an $x_{2}$ slice singular Borel probability measure on $[0,1)^{2}$ with slice measures $\{\gamma^{x_{1}}\}$ and marginal measure $\mu_{1}$ with corresponding auxiliary sequences $\{g_{n_{2}}^{x_{1}}\}$ and $\{g_{n_{1}}\}$. 
\begin{defn}
Let $\mathcal{V}_{\mu}:L^{2}(\mu)\to  H^{2}(\mathbb{D}^2)$
where $$\mathcal{V}_{\mu}(f)=\sum_{n_{2}=0}^{\infty}\sum_{n_{1}=0}^{\infty}\langle f(x_{1},x_{2}),g_{n_{2}}^{x_{1}}g_{n_{1}}\rangle_{\mu}z_{1}^{n_{1}} z_{2}^{n_{2}}.$$ We call $\mathcal{V}_{\mu}$ a \textbf{Normalized Cauchy Transform} for $\mu$.

Note this map is defined in terms of the sequence $\{g_{n_{2}}^{x_{1}}g_{n_{1}}\}$ just as the one dimensional Normalized Cauchy Transform can be defined.  Furthermore, as a result of Theorem $A$, $\mathcal{V}_{\mu}$ is a well-defined isometry.

Now if $\mu$ is slice singular in both directions, we have two Normalized Cauchy Transforms. If $\mu$ is $x_{1}$ slice singular with auxiliary sequences $\{h_{n_{2}}\}$ and $\{h_{n_{1}}^{x_{2}}\}$, we can define $\mathcal{V}_{\mu}:L^{2}(\mu)\to  H^{2}(\mathbb{D}^2)$
where $$\mathcal{V}_{\mu}(f)=\sum_{n_{1}=0}^{\infty}\sum_{n_{2}=0}^{\infty}\langle f(x_{1},x_{2}),h_{n_{1}}^{x_{2}}h_{n_{2}}\rangle_{\mu}z_{2}^{n_{2}} z_{1}^{n_{1}}.$$

We will denote $\mathcal{V}_{\mu}^{1}$ and $\mathcal{V}_{\mu}^{2}$ as two separate Normalized Cauchy Transforms defined in terms of different disintegrations when necessary for future discussion. We will actually show in this section that $\mathcal{V}_{\mu}^{1}=\mathcal{V}_{\mu}^{2}$ if and only if $\mu$ is a product measure.
\end{defn}

We can define a natural mapping that depends on the measurable field of measures $\{ \gamma^{x_1} \}$ that is also defined by the auxiliary sequence.
\begin{defn}
Let $\mathcal{V}_{\{\gamma^{x_{1}}\}}: L^{2}(\mu)\to L^{2}(\mu_{1})(z_{2})$ where
$$[\mathcal{V}_{\{\gamma^{x_{1}}\}}(f)](x_{1},z_{2})=\sum_{n_{2}=0}^{\infty}\langle f(x_{1},x_{2}), g_{n_{2}}^{x_{1}}(x_{2})\rangle_{\gamma^{x_{1}}}z_{2}^{n_{2}}.$$
Notice that $\mathcal{V}_{\{\gamma^{x_{1}}\}}$ is well defined isometry by equation (\ref{pf}).

Now define
$\widetilde{\mathcal{V}}_{\mu_{1}}: L^{2}(\mu_{1})(z_{2})\to \mathcal{H}(b(z_{1}))(z_{2})$
where $$[\widetilde{\mathcal{V}}_{\mu_{1}}(\sum_{n_{2}=0}^{\infty}a_{n_{2}}(x_{1})z_{2}^{n_{2}})](z_{2})=
\sum_{n_{2}=0}^{\infty}[\sum_{n_{1}=0}^{\infty}\langle a_{n_{2}},g_{n_{1}}\rangle_{\mu_{1}}z_{1}^{n_{1}}]z_{2}^{n_{2}}=\sum_{n_{2}=0}^{\infty}V_{\mu_{1}}(a_{n_{2}})z_{2}^{n_{2}}$$
where $b(z_{1})$ is the inner function corresponding to $\mu_{1}$.
Now because $V_{\mu_{1}}$ is a unitary onto $\mathcal{H}(b)$,  $\widetilde{\mathcal{V}}_{\mu_{1}}$ is a well defined unitary operator.
\end{defn}

The Normalized Cauchy Transform in the one dimensional case is normally discussed in an integral form. Because of this, one can use the Rokhlin disintegration to prove the following Lemma:
\begin{lem}
$\mathcal{V}_{\mu}$ is an isometry, and for each $z_{1},z_{2}\in \mathbb{D}$ and $f\in L^{2}(\mu)$,
    $$[\mathcal{V}_{\mu}(f)](z_{1},z_{2})=[\widetilde{\mathcal{V}}_{\mu_{1}}\circ \mathcal{V}_{\{\gamma^{x_{1}}\}}(f)](z_{1},z_{2})=
\int_{0}^{1}\int_{0}^{1}\frac{fd\gamma^{x_{1}}d\mu_{1}}{(1-z_{2}e^{-2\pi i x_{2}})(1-z_{1}e^{-2\pi i x_{1}})[C_{\gamma^{x_{1}}}(1)](z_2)[C_{\mu_{1}}(1)](z_1)}.$$
\end{lem}

\subsection{Images of Normalized Cauchy Transforms}
Now we characterize the image of our two-dimensional Normalized Cauchy Transform.
\begin{defn}
Let 
$$B(z_{2})=\sum_{n_{2}=0}^{\infty}B_{n_{2}}z_{2}^{n_{2}}$$ be the analytic operator acting isometrically on $L^{2}(\mu_{1})(z_{2})$ as in the proof of Theorem \ref{Th:main1}.
Now we define
$$\widetilde{B}(z_{2})=: \sum_{n_{2}=0}^{\infty}V_{\mu_{1}}B_{n_{2}}V_{\mu_{1}}^{*}z_{2}^{n_{2}}$$
is an analytic operator acting on $\mathcal{H}(b(z_{1}))(z_{2})$ by power series multiplication.
We have that $\widetilde{B}(z_{2})$ acts as an isometry.
\end{defn}
Applying the Theorem of de Branges on Hilbert space $\mathcal{H}(b)$ with analytic operator $\widetilde{B}(z_{2})$ acting on $\mathcal{H}(b(z_{1}))(z_{2})$, we have unitary $U$ acting on 
$\mathcal{H}(\widetilde{B})=\overline{span}\{U^{-n}(c): c\in \mathcal{H}(b(z_{1})), n\in \mathbb{N}\}$ such that
$$U^{*}(f(z_{2}))=\frac{f(z_{2})-f(0)}{z_{2}}+\frac{\widetilde{B}(z_{2})}{z_{2}}f(0).$$
Now thinking of $\widetilde{B}(z_{2})$ as our inner function in two-dimensional space, we produce the following:

\begin{thm} \label{Th:NCT-image}
Let $\mu$ be an $x_{2}$ slice singular Borel probability measure on $[0,1)^{2}$,
$$Im(\mathcal{V}_{\mu})=\mathcal{H}(\widetilde{B}).$$
\end{thm}
\begin{proof}
Notice that for any $g\in L^{2}(\mu_{1})$,
\begin{equation}
\begin{split}{\label{firstinv}}
    \frac{\widetilde{B}(z_{2})}{z_{2}}[V_{\mu_{1}}(g)] & =\widetilde{\mathcal{V}}_{\mu_{1}} \frac{B(z_{2})}{z_{2}}\widetilde{\mathcal{V}}_{\mu_{1}}^{*}[V_{\mu_{1}}(g)] \\
    & =\widetilde{\mathcal{V}}_{\mu_{1}} M_{\mathcal{V}_{\{\gamma^{x_{1}}\}}(e^{-2\pi i x_{2}})}\widetilde{\mathcal{V}}_{\mu_{1}}^{*}[V_{\mu_{1}}(g)] \\
    &= \widetilde{\mathcal{V}}_{\mu_{1}} M_{\mathcal{V}_{\{\gamma^{x_{1}}\}}(e^{-2\pi i x_{2}})}[g]\\
    &= \widetilde{\mathcal{V}}_{\mu_{1}}\mathcal{V}_{\{\gamma^{x_{1}}\}}(e^{-2\pi i x_{2}}g)\\
    & =\mathcal{V}_{\mu}(e^{-2\pi i x_{2}}g)
\end{split}
\end{equation}
where $M_{\mathcal{V}_{\{\gamma^{x_{1}}\}}(e^{-2\pi i x_{2}})}$ denotes multiplication by $\mathcal{V}_{\{\gamma^{x_{1}}\}}(e^{-2\pi i x_{2}})$ on $L^{2}(\mu_{1})(z_{2})$ and the second equality follows from equation \ref{identityifs}.

Furthermore,
for any $f\in L^{2}(\mu)$,
\begin{equation}{\label{secondinv}}
\begin{split}
\frac{[\mathcal{V}_{\mu}(f)](z_{2})-[\mathcal{V}_{\mu}(f)](0)}{z_{2}} & =
\widetilde{\mathcal{V}}_{\mu_{1}}(\frac{[\mathcal{V}_{\{\gamma^{x_{1}}\}}(f)](z_{2})-[\mathcal{V}_{\{\gamma^{x_{1}}\}}(f)](0)}{z_{2}})\\
& =\widetilde{\mathcal{V}}_{\mu_{1}} \mathcal{V}_{\{\gamma^{x_{1}}\}}(e^{-2\pi i x_{2}}(f-\langle f, 1\rangle_{\gamma^{x_{1}}}))\\
& =\mathcal{V}_{\mu}(e^{-2\pi i x_{2}}(f-\langle f, 1\rangle_{\gamma^{x_{1}}}))
\end{split}
\end{equation}
where the second equality follows from equation (\ref{backward}).
Therefore, we have
$$Im(\mathcal{V}_{\mu}) \supseteq \mathcal{H}(\widetilde{B})$$
since $Im(\mathcal{V}_{\mu}) \supseteq \mathcal{H}(b)$.

Now suppose that for $f\in L^{2}(\mu)$,
$$\langle \mathcal{V}_{\mu}(f), U^{-n}(c)\rangle=0$$ for all $c\in H(b(z_{1}))$ and $n\in \mathbb{N}$.

Letting $n=0$, we have,
$$\langle f, g\rangle=\langle \mathcal{V}_{\mu}(f), \mathcal{V}_{\mu}(g)\rangle=0$$ for all $g\in L^{2}(\mu_{1})$.

Notice that by equation (\ref{firstinv}),
$$U^{-1}(\mathcal{V}_{\mu}(g))=\mathcal{V}_{\mu}(e^{-2\pi i x_{2}}g).$$
Therefore,
$$\langle f, e^{-2\pi ix_{2}}g\rangle=0$$ for all $g\in L^{2}(\mu_{1}).$

Furthermore, for any $k$ and $g\in L^{2}(\mu)$ by equation (\ref{firstinv}) and (\ref{secondinv}),
\begin{equation}
\begin{split}
U^{-1}(\mathcal{V}_{\mu}(e^{-2\pi i kx_{2}}g))=
\mathcal{V}_{\mu}(e^{-2\pi ix_{2}}(e^{-2\pi ikx_{2}}g-\langle e^{-2\pi ikx_{2}}g,1\rangle_{\gamma^{x_{1}}}))+[\frac{\widetilde{B}(z_{1})}{z_{1}}\mathcal{V}_{\mu}(e^{-2\pi i kx_{2}}g)](0).
\end{split}
\end{equation}
Then it is easy to see by induction on $n\geq 2$,
$$U^{-n}(\mathcal{V}_{\mu}(g))=\mathcal{V}_{\mu}(e^{-2\pi i nx_{2}}g)+\sum_{k=1}^{n-1}\mathcal{V}_{\mu}(e^{-2\pi i kx_{2}}f_{k})$$
where $f_{k}\in L^{2}(\mu_{1}).$

Therefore we have,
$$\langle f, e^{-2\pi i nx_{2}}g\rangle=0$$
for all $n\in \mathbb{N}$ and $g\in L^{2}(\mu_{1}).$

Then we have
$$\langle \overline{f},ge^{2\pi i nx_{2}}\rangle=0$$ for all $n\in \mathbb{N}$ and $g\in L^{2}(\mu_{1})$.
By the Rokhlin disintegration,
we have
$$\langle \langle \overline{f},e^{2\pi i nx_{2}}\rangle_{\gamma^{x_{1}}},g\rangle_{\mu_{1}}=0\implies
\langle \overline{f},e^{2\pi i nx_{2}}\rangle_{\gamma^{x_{1}}}(x_{1})=0 .$$
Therefore, for $\mu_{1}$ a.e. $x_{1}$, $$f(x_{1},x_{2})=0$$ by the completeness of $\{e^{2\pi i x_{2}}\}$ in $L^{2}(\gamma^{x_{1}})$. Then $f=0$, getting us that
$$Im(\mathcal{V}_{\mu})\subseteq \mathcal{H}(\widetilde{B}).$$

\end{proof}
Now by a similar proof we have the following:
\begin{cor}
$$Im(\mathcal{V}_{\{\gamma^{x_{1}}\}})=\mathcal{H}(B)=\int_{[0,1)^{d-1}}^{\oplus}\mathcal{H}(b^{x_{1},\dots x_{d-1}})d\mu_{1}$$ so that
$$Im(\mathcal{V}_{\mu})=\widetilde{\mathcal{V}}_{\mu_{1}}(\mathcal{H}(B))=\widetilde{\mathcal{V}_{\mu_{1}}}(\int_{[0,1)^{d-1}}^{\oplus}\mathcal{H}(b^{x_{1},\dots x_{d-1}})d\mu_{1}).$$
\end{cor}

Now naturally one asks if these two-dimensional Normalized Cauchy Transforms are the same for a measure $\mu$ that is slice singular in both directions. We now show that only product measures can have the same Normalized Cauchy Transforms. In fact, they are the only measures whose Normalized Cauchy Transforms have the same image.

\begin{thm}
Let $\mu$ a Borel probability measure on $[0,1)^{2}$ that is slice singular in any variable order with $x_{2}$ slice singular auxiliary sequences
$\{g_{n_{2}}^{x_{1}}\}$ and $\{g_{n_{1}}\}$ and $x_{1}$ slice singular auxiliary sequences $\{h_{n_{1}}^{x_{2}}\}$ and $\{h_{n_{2}}\}$. The following are equivalent:
\begin{enumerate}
    \item $\mathcal{V}_{\mu}^{1}=\mathcal{V}_{\mu}^{2}$
    \item $Im(\mathcal{V}_{\mu}^{1})=Im(\mathcal{V}_{\mu}^{2})$
    \item The sequences $\{g_{n_{2}}^{x_{1}}g_{n_{1}}\}$ and $\{h^{x_{2}}_{n_{1}}h_{n_{2}}\}$ are unitarily equivalent.
    \item $\mu$ is a product measure.
\end{enumerate}
\end{thm}
\begin{proof}
$(1)\implies (2) $ is obvious.

For $(4)\implies (1)$, if $\mu$ is a product measure, it is clear $h_{n_{2}}=g_{n_{2}}^{x_{1}}$ and $g_{n_{1}}=h_{n_{1}}^{x_{2}}$ in $L^{2}(\mu)$ for each $n_{1}$ and $n_{2}$. The equality of $\mathcal{V}_{\mu}^{1}$ and $\mathcal{V}_{\mu}^{2}$ follow.

For $(2)\implies (3)$, we have for all $f\in L^{2}(\mu)$, there is a corresponding $U^{*}(f)\in L^{2}(\mu)$ s.t.
    $$\sum_{n_{2}=0}^{\infty}[\sum_{n_{1}=0}^{\infty}\langle f,g_{n_{2}}^{x_{1}}g_{n_{1}}\rangle_{\mu}z_{1}^{n_{1}}]z_{2}^{n_{2}}=\sum_{n_{1}=0}^{\infty}[\sum_{n_{2}=0}^{\infty}\langle U^{*}(f),h_{n_{1}}^{x_{2}}h_{n_{2}}\rangle z_{2}^{n_{2}}]z_{1}^{n_{1}}.$$
    Then it is clear that $U^{*}$ is a bijection.
    Now because, $\{g_{n_{2}}^{x_{1}}g_{n_{1}}\}$ and $\{h_{n_{1}}^{x_{2}}h_{n_{2}}\}$ are Parseval frames in $L^{2}(\mu)$, we have $$||U^{*}(f)||=||f||$$ for all $f\in L^{2}(\mu)$.
    Also, by linearity of the inner product, we have $U^{*}$ is linear.

For $(3)\implies (4)$, let $U$ be the unitary in condition (3). Note that by taking $n_2=0$, we have $U(g_k)=h_k^{x_2}$ for all $k$.
    We prove that $\int_{0}^{1}e^{2\pi i n_{1}x_{1}}d\mu_{1}=\int_{0}^{1}e^{2\pi i n_{1}x_{1}}d\rho^{x_{2}}$ in $L^{2}(\mu)$ for all $n_{1}$ by induction where $\mu_{1}$ is the marginal measure on $x_{1}$, $\mu_{2}$ is the marginal measure on $x_{2}$, and $\rho^{x_{2}}$ are the slice measures of $\mu$ that are indexed by $x_{2}$.

The $n_{1}=0$, case follows from $\mu_{1}$ and $\rho^{x_{2}}$ being probability measures.

Assume the result holds for all $k\leq n_{1}-1$. Taking Equation (\ref{Eq:aux}) in $L^2(\mu_1)$, and then recognizing it as holding in $L^2(\mu)$, by applying $U$ we have
\begin{align*}
U(e_{n_1})&=\sum_{k=0}^{n_1}\ip{e_{n_1},e_k}_{\mu_1}U(g_k)\\
&=\sum_{k=0}^{n_1}\ip{e_{n_1},e_k}_{\mu_1}h_k^{x_2}.
\end{align*}
On the other hand, by applying Equation (\ref{Eq:aux}) directly in $L^2(\rho^{x_2})$, we have
\begin{align*}e_{n_1}&=\sum_{k=0}^{n_1}\ip{e_{n_1},e_k}_{\rho^{x_2}}h_k^{x_2}
\end{align*}
for $\mu_2$-almost-every $x_2$. Thus this equation also holds in $L^2(\mu)$. It follows from the induction assumption that 
\begin{align*}U(e_{n_1})=e_{n_1}+\ip{e_{n_1},1}_{\mu_1}-\ip{e_{n_1},1}_{\rho^{x_2}}.
\end{align*}

Taking the $L^2(\mu)$ norm of both sides,

\begin{align*}
1&= \norm{e^{2\pi i n_{1}x_{1}}+\langle e^{2\pi i n_{1}x_{1}},1\rangle_{\mu_{1}}-\langle e^{2\pi i n_{1}x_{1}},1\rangle_{\rho^{x_{2}}}}_{\mu}^{2} \\
& = 1+ \norm{\langle e^{2\pi i n_{1}x_{1}},1\rangle_{\mu_{1}}-\langle e^{2\pi i n_{1}x_{1}},1\rangle_{\rho^{x_{2}}}}_{\mu}^{2}+2Re(\langle e^{2\pi i n_{1}x_{1}}, \langle e^{2\pi i n_{1}x_{1}},1\rangle_{\mu_{1}}-\langle e^{2\pi i n_{1}x_{1}},1\rangle_{\rho^{x_{2}}}\rangle_{\mu})\\
&=1+\norm{\langle e^{2\pi i n_{1}x_{1}},1\rangle_{\mu_{1}}-\langle e^{2\pi i n_{1}x_{1}},1\rangle_{\rho^{x_{2}}}}_{\mu}^{2} +2\p{|\langle e^{2\pi i n_{1}x_{1}},1\rangle_{\mu_{1}}|^{2}-\int_{0}^{1}\abs{\int_{0}^{1}e^{2\pi i n_{1}x_{1}}d\rho^{x_{2}}}^{2}d\mu_{2}} \\
&=1+|\langle e^{2\pi i n_{1}x_{1}},1\rangle_{\mu_{1}}|^{2}+\int_{0}^{1}\abs{\int_{0}^{1}e^{2\pi i n_{1}x_{1}}d\rho^{x_{2}}}^{2}d\mu_{2}-2Re(\langle e^{2\pi i n_{1}x_{1}},1\rangle_{\mu_{1}},\langle e^{2\pi i n_{1}x_{1}},1\rangle_{\rho^{x_{2}}})\\
& \hspace{0.65cm}+2\p{|\langle e^{2\pi i n_{1}x_{1}},1\rangle_{\mu_{1}}|^{2}-\int_{0}^{1}\abs{\int_{0}^{1}e^{2\pi i n_{1}x_{1}}d\rho^{x_{2}}}^{2}d\mu_{2}} \\
& = 1+|\langle e^{2\pi i n_{1}x_{1}},1\rangle_{\mu_{1}}|^{2}-\int_{0}^{1}\abs{\int_{0}^{1}e^{2\pi i n_{1}x_{1}}d\rho^{x_{2}}}^{2}d\mu_{2} \\
&=1+Re(\langle e^{2\pi i n_{1}x_{1}}, \langle e^{2\pi i n_{1}x_{1}},1\rangle_{\mu_{1}}-\langle e^{2\pi i n_{1}x_{1}},1\rangle_{\rho^{x_{2}}}\rangle_{\mu}).
\end{align*}

This implies that $Re(\langle e^{2\pi i n_{1}x_{1}}, \langle e^{2\pi i n_{1}x_{1}},1\rangle_{\mu_{1}}-\langle e^{2\pi i n_{1}x_{1}},1\rangle_{\rho^{x_{2}}}\rangle_{\mu})=0$, which substituting back yields 
\begin{align*}||\langle e^{2\pi i n_{1}x_{1}},1\rangle_{\mu_{1}}-\langle e^{2\pi i n_{1}x_{1}},1\rangle_{\rho^{x_{2}}}||_{\mu}^{2}=0.
\end{align*}
Therefore, we have for $\mu_{2}$-almost-every $x_{2}$,
$$\int_{0}^{1}e^{2\pi i n_{1}x_{1}}d[\mu_{1}-\rho^{x_{2}}]=0$$ for all $n_{1}$. Then by the F. and M. Riesz theorem, $\mu_1-\rho^{x_2}$ is absolutely continuous, and since it is also singular, $\mu_1-\rho^{x_2}$ is the zero measure. Thus we have $\mu_{1}=\rho^{x_{2}}$ for $\mu_{2}$-almost-every $x_{2}$, so that by the uniqueness part of the Rokhlin Disintegration, $\mu$ is a product measure.

\end{proof}

\subsection{Symmetric measures}
Now we discuss a special case of the Normalized Cauchy Transform for symmetric measures. One might initially believe that a symmetric measure has its Normalized Cauchy Transforms with the same images, but we have already shown this is not the case. We will show however that the images are reflections of each other.

Suppose that $\mu$ is a symmetric slice singular Borel probability measure on $[0,1)^{2}$, meaning the map $T$ on $L^{2}(\mu)$ where
$$T(f(x_{1},x_{2}))=f(x_{2},x_{1})$$ is an isometry.

Let $\{g_{n_{1}}(x_{1})\}$ be the auxiliary sequence associated with the marginal measure $\mu_{1}$ of $\mu$ on variable $x_{1}$ and let $\{h_{n_{1}}(x_{2})\}$ be the auxiliary sequence associated with the marginal measure $\mu_{2}$ of $\mu$ on variable $x_{2}$. Using the recursive definitions, one can see that
$$T(h_{n_{1}}(x_{2}))=g_{n_{1}}(x_{1})$$ for all $n_{1}$ since
$$\langle e^{2\pi i n_{1}x_{1}},1\rangle_{\mu_{1}}=\langle e^{2\pi i n_{1}x_{2}},1\rangle_{\mu_{2}}.$$

Then for all $n_{1}$, $g_{n_{1}}(x_{1})=h_{n_{1}}(x_{1})$ and similarly, for $\{g_{n_{2}}^{x_{1}}(x_{2})\}$, the auxiliary sequence associated with the slice measures $\{\gamma^{x_{1}}\}$ of $\mu$ that are in variable $x_{2}$ and auxiliary sequence $\{h_{n_{2}}^{x_{2}}(x_{1})\}$ associated with the slice measures $\{\rho^{x_{2}}\}$ of variable $x_{1}$ one can show
$$T(h_{n_{2}}^{x_{2}}(x_{1}))=g_{n_{2}}^{x_{1}}(x_{2})$$ for all $n_{2}$ since
$$T(\langle e^{2\pi i n_{2}x_{1}},1\rangle_{\rho^{x_{2}}})=\langle e^{2\pi i n_{2}x_{2}},1\rangle_{\gamma^{x_{1}}}$$ for all $n_{2}$.

Then we have for any $f\in L^{2}(\mu)$ and any $n_{1}$ and $n_{2}$
$$\langle f(x_{1},x_{2}),g_{n_{2}}^{x_{1}}(x_{2})g_{n_{1}}(x_{1})\rangle =\langle f(x_{1},x_{2}),h_{n_{2}}^{x_{1}}(x_{2})h_{n_{1}}(x_{1})\rangle=
\langle f(x_{2},x_{1}),h_{n_{2}}^{x_{2}}(x_{1})h_{n_{1}}(x_{2})\rangle.$$

Now let $T'$ be the unitary map on $H^{2}(\mathbb{D}^{2})$ where
$$T'(\sum_{n_{1}}\sum_{n_{2}}c_{n_{1}n_{2}}z_{1}^{n_{1}}z_{2}^{n_{2}})=
\sum_{n_{1}}\sum_{n_{2}}c_{n_{2}n_{1}}z_{1}^{n_{1}}z_{2}^{n_{2}}.$$

We have the following proposition describing the Normalized Cauchy Transforms of a symmetric measure.
\begin{prop}
Let $\mu$ be a symmetric Borel probability measure on $[0,1)^{2}$ that is slice singular in any variable order with Normalized Cauchy Transforms $\mathcal{V}_{\mu}^{1}$ and $\mathcal{V}_{\mu}^{2}$.
We have
$$\mathcal{V}_{\mu}^{1}=T' \circ \mathcal{V}_{\mu}^{2} \circ T$$ 
so that
$$Im(\mathcal{V}_{\mu}^{1})=T'(Im(\mathcal{V}_{\mu}^{2})).$$
\end{prop}

\subsection{Boundary behavior}
We now attempt to justify why our definition of Normalized Cauchy Transform plays nicer with product measures than symmetric measures by showing
that the two dimensional Normalized Cauchy Transform of a function $f$ converges to $f$ in norm as it approaches the boundary in some sense. Namely,

\begin{prop}
Let $\mathcal{V}_{\mu}$ be defined where $\mu$ is $x_{2}$ slice singular. Then for any $f\in L^{2}(\mu)$,
$$\lim_{r_{2}\to 1}\lim_{r_{1}\to 1}[\mathcal{V}_{\mu}(f)](r_{1}e^{2\pi ix_{1}},r_{2}e^{2\pi i x_{2}})=f(x_{1},x_{2}).$$
where each limit is taken in $L^{2}(\mu)$ and the first limit is taken with fixed $r_{2}$ and $x_{2}$, followed by the second limit.
\end{prop}
There is an analogous statement when $\mu$ is $x_{1}$ slice singular, but it is important that the limit with the variable associated with the marginal measure is taken first.
\begin{proof}
Let $z_{1}=r_{1}e^{2\pi i x_{1}}$ and $z_{2}=r_{2}e^{2\pi i x_{2}}$.
Recall for all $f\in L^{2}(\mu)$,
$$[\mathcal{V}_{\mu}(f)](z_{2})=[\tilde{\mathcal{V}}_{\mu_{1}}\mathcal{V}_{\{\gamma^{x_{1}}\}}(f)](z_{2})=\sum_{n_{2}=0}^{\infty}V_{\mu_{1}}(\langle f, g_{n_{2}}^{x_{1}}\rangle_{\gamma^{x_{1}}})z_{2}^{n_{2}}.$$
For fixed $z_{2}$, since each component of the composition is an isometry, we have absolute summability so that
$$\lim_{|z_{1}|\to 1}\mathcal{V}_{\mu}(f)=\sum_{n_{2}=0}^{\infty}\lim_{|z_{1}|\to 1}V_{\mu_{1}}(\langle f, g_{n_{2}}^{x_{1}}\rangle_{\gamma^{x_{1}}})z_{2}^{n_{2}}.$$

Now $$[V_{\mu_{1}}(\langle f, g_{n_{2}}^{x_{1}}\rangle_{\gamma^{x_{1}}})](e^{2\pi i n_{1}x_{1}})=\sum_{n_{1}=0}^{\infty}\langle \langle f, g_{n_{2}}^{x_{1}}\rangle_{\gamma^{x_{1}}},g_{n_{1}}\rangle_{\mu_{1}}e^{2\pi i n_{1}x_{1}}= \langle f, g_{n_{2}}^{x_{1}}\rangle_{\gamma^{x_{1}}}$$ by the one dimensional case. Therefore, by Abel summability for Hilbert spaces,
$$\lim_{|z_{1}|\to 1}\mathcal{V}_{\mu}(f)=\sum_{n_{2}=0}^{\infty}\langle f, g_{n_{2}}^{x_{1}}\rangle_{\gamma^{x_{1}}}z_{2}^{n_{2}}.$$

Finally, recall from the proof of Theorem A,
$$f=\sum_{n_{2}=0}^{\infty}[G(n_{2},d)f]e^{2\pi i n_{2}x_{2}}=\sum_{n_{2}=0}^{\infty}\langle f, g_{n_{2}}^{x_{1}}\rangle_{\gamma^{x_{1}}}e^{2\pi i n_{2}x_{2}}.$$
Then we may apply Abel summability again to get the result.

\end{proof}

\subsection{$d$ dimensional Normalized Cauchy Transform}
In this section we discuss how one could define a Normalized Cauchy Transform in higher dimensions. In particular, we illustrate how this transform can be disintegrated into $d$ natural parts.

\begin{defn}
If $\mu$ is an $x_{d}$ slice singular Borel probability measure on $[0,1)^{d}$, we can naturally define a \textbf{Normalized Cauchy Transform} of $\mu$ as follows: $\mathcal{V}_{\mu}: L^{2}(\mu)\to H^{2}(\mathbb{D}^{d})$ where
$$\mathcal{V}_{\mu}(f)=\sum_{n_{d}=0}^{\infty}\sum_{n_{d-1}=0}^{\infty}\dots \sum_{n_{1}=0}^{\infty}c_{n_{1},\dots,n_{d}}z_{1}^{n_{1}}\dots z_{d-1}^{n_{d-1}}z_{d}^{n_{d}}$$ where the sequence $\{c_{n_{1},\dots,n_{d}}\}$ is given in Theorem A.

By Theorem $A$, this map is an isometry.

Furthermore, we can similarly define $\mathcal{V}_{\{\gamma^{x_{1},\dots,x_{d-1}}\}}:L^{2}(\mu)\to L^{2}(\mu_{1})(z_{d})$ where
$$\mathcal{V}_{\{\gamma^{x_{1},\dots,x_{d-1}}\}}(f)=\sum_{n_{d}=0}^{\infty}\langle f,g_{n_{d}}^{x_{1},\dots,x_{d-1}}(x_{d})\rangle_{\gamma^{x_{1},\dots, x_{d-1}}}z_{d}^{n_{d}}$$ and $\widetilde{\mathcal{V}}_{\mu_{1}}: L^{2}(\mu_{1})(z_{d})\to Im(\mathcal{V}_{\mu_{1}})(z_{d})$ inductively where
$$\widetilde{\mathcal{V}}_{\mu_{1}}(\sum_{n_{d}=0}^{\infty}f_{n_{d}}(x_{1},\dots, x_{d-1})z_{d}^{n_{d}})=\sum_{n_{d}=0}^{\infty}\mathcal{V}_{\mu_{1}}(f_{n_{d}}(x_{1},\dots, x_{d-1}))z_{d}^{n_{d}}$$ and $\mathcal{V}_{\mu_{1}}: L^{2}(\mu_{1})\to H^{2}(\mathbb{D}^{d-1})$ is the Normalized Cauchy Transform defined inductively  for $\mu_{1}$ that is $x_{d-1}$ slice singular.
\end{defn}
Again we have that $\mathcal{V}_{\mu}=\tilde{\mathcal{V}_{\mu_{1}}}\circ \mathcal{V}_{\{\gamma^{x_{1},\dots,x_{d-1}}\}}$.

By an induction argument, we get the following:

\begin{prop}
If $\mu$ is an $x_{d}$ slice singular Borel probability measure on $[0,1)^{d}$, there exists $\widetilde{\mathcal{V}}_{\mu_{d-1}}: L^{2}(\mu_{d-1})(z_{2},\dots, z_{d})\to H^{2}(\mathbb{D}^{d}) $ where 
$$\widetilde{\mathcal{V}}_{\mu_{d-1}}(\sum_{n_{d}}\dots \sum_{n_{2}}f_{n_{2},\dots ,n_{d}}(x_{1})z_{2}^{n_{2}}\dots z_{d}^{n_{d}})=\sum_{n_{d}}\dots \sum_{n_{2}}\mathcal{V}_{\mu_{d-1}}(f_{n_{2},\dots ,n_{d}}(x_{1}))z_{2}^{n_{2}}\dots z_{d}^{n_{d}}$$ and $\mathcal{V}_{\mu_{d-1}}: L^{2}(\mu_{d-1})\to H^{2}(\mathbb{D})$ is the Normalized Cauchy Transform of $\mu_{d-1}$ and $\mathcal{V}_{\{\gamma^{x_{1},\dots x_{k}}\}}$ where $1\leq k\leq d-2$ that are defined on $L^{2}(\mu_{k})(z_{k+2},\dots, z_{d} )$ so that
$$\mathcal{V}_{\mu}=\widetilde{\mathcal{V}}_{\mu_{d-1}}\circ \mathcal{V}_{\{\gamma^{x_{1}}\}}\circ \dots \circ \mathcal{V}_{\{\gamma^{x_{1},\dots x_{d-2}}\}} \circ \mathcal{V}_{\{\gamma^{x_{1},\dots x_{d-1}}\}}.$$
\end{prop}

Notice again that if $\mu$ is slice singular in any variable order, there are $d!$ ways to define $\mathcal{V}_{\mu}$ in terms of the auxiliary sequences.

To illustrate this, we give a commutative diagram for the case when $d=3$.

\vspace{1cm}

\adjustbox{scale=1.5,center}{
\begin{tikzcd}
L^{2}(\mu) \arrow[rd, "\mathcal{V}_{\mu}"] \arrow[r, "{\mathcal{V}_{\{\gamma^{x_{1},x_{2}}\}}}"] & L^{2}(\mu_{1})(z_{3}) \arrow[r, "{\mathcal{V}_{\{\gamma^{x_{1}}\}}}"] \arrow[d, "\widetilde{V}_{\mu_{1}}"] & {L^{2}(\mu_{2})(z_{2},z_{3})} \arrow[d, "\widetilde{V}_{\mu_{2}}"] \\
                                                                                               & H^{2}(\mathbb{D}^{2})(z_{3})                                                                         & H^{2}(\mathbb{D}^{3})
                                                                                               \arrow[l, "i"]                                                   
\end{tikzcd}
}
Where $i$ is the natural inclusion map.

We can also illustrate more from looking at the specific images of some of these maps.

\vspace{1cm}

\adjustbox{scale=1.5,center}{
\begin{tikzcd}
L^{2}(\mu) \arrow[rd, "\mathcal{V}_{\mu}"] \arrow[r, "{\mathcal{V}_{\{\gamma^{x_{1},x_{2}}\}}}"] & L^{2}(\mu_{1})(z_{3}) \arrow[r, "{\mathcal{V}_{\{\gamma^{x_{1}}\}}}"] \arrow[d, "\widetilde{V}_{\mu_{1}}"] & {L^{2}(\mu_{2})(z_{2},z_{3})} \arrow[d, "\widetilde{V}_{\mu_{2}}"] \\
                                                                                               & \mathcal{H}(\widetilde{B})(z_{3})                                                                         & {\mathcal{H}(b_{2})(z_{2},z_{3}) \arrow[l, "P"]                         }                          
\end{tikzcd}
}

Where $b_{2}$ is in inner function corresponding to $\mu_{2}$
and 
$P: [H(b_{2})(z_{2})](z_{3})\to \mathcal{H}(\widetilde{B})(z_{3})$ is the projection map, thinking of $\mathcal{H}(\widetilde{B})$ as a closed subspace of $H(b_{2})(z_{2})$, specifically,
$$[\widetilde{B}(z_{2})[\mathcal{H}(b_{2})(z_{2})]]^{\perp}=\mathcal{H}(\widetilde{B})=Im(\mathcal{V}_{\mu_{1}})$$ where $\widetilde{B}(z_{2})$ is defined in the proof of Theorem A for $\mu_{1}$.

Now let
$$\widetilde{\widetilde{B}}(z_{3})=\widetilde{V}_{\mu_{1}}\circ B(z_{3})\circ \widetilde{V}_{\mu_{1}}^{-1}$$
where $B(z_{3})$ is defined in the proof for Theorem A.
\begin{prop}

We have
$$Im(\mathcal{V}_{\mu})=H(\widetilde{\widetilde{B}})={[\widetilde{\widetilde{B}}(z_{3})H(\widetilde{B})(z_{3})]^{\perp}}={[\widetilde{\widetilde{B}}(z_{3})[[\widetilde{B}(z_{2})[\mathcal{H}(b_{2})(z_{2})]]^{\perp}](z_{3})]^{\perp}}.$$
\end{prop}
Note for $d>3$ the image of the Normalized Cauchy Transform and the  properties of the transform discussed in the last section follow by inductive arguments. The trapezoidal commutative diagram above extends as well by decomposing $\widetilde{V}_{\mu_{2}}$ similarly.

\section*{Acknowledgments}
This research was supported in part by the National Science Foundation and the National Geospatial Intelligence Agency under awards \#1830254 and \#2219959.

\printbibliography 

@Article{Herr2017Fourier,
AUTHOR = {Herr, John E. and Weber, Eric S.},
TITLE = {Fourier Series for Singular Measures},
JOURNAL = {Axioms},
VOLUME = {6},
YEAR = {2017},
NUMBER = {2},
ARTICLE-NUMBER = {7},
URL = {https://www.mdpi.com/2075-1680/6/2/7},
ISSN = {2075-1680},
}

@incollection {Herr2022Fourier,
    AUTHOR = {Herr, John E. and Jorgensen, Palle E. T. and Weber, Eric S.},
     TITLE = {Fourier Series for Fractals in Two Dimensions},
 BOOKTITLE = {From classical analysis to analysis on fractals},
    SERIES = {Appl. Numer. Harmon. Anal.},
     PAGES = {183--229},
 PUBLISHER = {Birkh\"{a}user/Springer, Cham},
      YEAR = {2023},
      ISBN = {978-3-031-37800-3},
   MRCLASS = {42 (28)},
  MRNUMBER = {4676388},
       DOI = {10.1007/978-3-031-37800-3\_9},
       URL = {https://doi.org/10.1007/978-3-031-37800-3_9},
}

@article{Kwapien2001Kaczmarz,
author = {Stanisław Kwapień and Jan Mycielski},
journal = {Studia Mathematica},
language = {eng},
number = {1},
pages = {75-86},
title = {On the Kaczmarz algorithm of approximation in infinite-dimensional spaces},
url = {http://eudml.org/doc/284583},
volume = {148},
year = {2001},
}

@article{Haller2005Kaczmarz,
author = {Rainis Haller and Ryszard Szwarc},
journal = {Studia Mathematica},
language = {eng},
number = {2},
pages = {123-132},
title = {Kaczmarz algorithm in Hilbert space},
url = {http://eudml.org/doc/284674},
volume = {169},
year = {2005},
}

@misc{Garrett1996Good,
url = {https://www-users.cse.umn.edu/~garrett/m/fun/good_spectral_thm.pdf},

author = {Paul Garrett},

title= {A Good Spectral Theorem},

year= {1996},
}

@article {DeBranges1960Some,
    AUTHOR = {de Branges, Louis},
     TITLE = {Some Hilbert spaces of entire functions},
   JOURNAL = {Trans. Amer. Math. Soc.},
  FJOURNAL = {Transactions of the American Mathematical Society},
    VOLUME = {96},
      YEAR = {1960},
     PAGES = {259--295},
      ISSN = {0002-9947,1088-6850},
   MRCLASS = {30.85},
  MRNUMBER = {133455},
       DOI = {10.2307/1993464},
       URL = {https://doi.org/10.2307/1993464},
}

@article {Jorgensen1998Dense,
    AUTHOR = {Jorgensen, Palle E. T. and Pedersen, Steen},
     TITLE = {Dense analytic subspaces in fractal $L^2$ spaces},
   JOURNAL = {J. Anal. Math.},
  FJOURNAL = {Journal d'Analyse Math\'{e}matique},
    VOLUME = {75},
      YEAR = {1998},
     PAGES = {185--228},
      ISSN = {0021-7670,1565-8538},
   MRCLASS = {46E30 (28A75 42C05 46L55 47B38)},
  MRNUMBER = {1655831},
MRREVIEWER = {Javier\ Soria},
       DOI = {10.1007/BF02788699},
       URL = {https://doi.org/10.1007/BF02788699},
}

@article {He2013Exponential,
    AUTHOR = {He, Xing-Gang and Lai, Chun-Kit and Lau, Ka-Sing},
     TITLE = {Exponential spectra in $L^2(\mu)$},
   JOURNAL = {Appl. Comput. Harmon. Anal.},
  FJOURNAL = {Applied and Computational Harmonic Analysis. Time-Frequency
              and Time-Scale Analysis, Wavelets, Numerical Algorithms, and
              Applications},
    VOLUME = {34},
      YEAR = {2013},
    NUMBER = {3},
     PAGES = {327--338},
      ISSN = {1063-5203,1096-603X},
   MRCLASS = {42C15 (46E30)},
  MRNUMBER = {3027906},
MRREVIEWER = {Keri\ A.\ Kornelson},
       DOI = {10.1016/j.acha.2012.05.003},
       URL = {https://doi.org/10.1016/j.acha.2012.05.003},
}

@article {Picioroaga2017Fourier,
    AUTHOR = {Picioroaga, Gabriel and Weber, Eric S.},
     TITLE = {Fourier frames for the Cantor-4 set},
   JOURNAL = {J. Fourier Anal. Appl.},
  FJOURNAL = {The Journal of Fourier Analysis and Applications},
    VOLUME = {23},
      YEAR = {2017},
    NUMBER = {2},
     PAGES = {324--343},
      ISSN = {1069-5869,1531-5851},
   MRCLASS = {42C15 (28A80 42B05 46L89)},
  MRNUMBER = {3622655},
MRREVIEWER = {Krishnan\ Parthasarathy},
       DOI = {10.1007/s00041-016-9471-0},
       URL = {https://doi.org/10.1007/s00041-016-9471-0},
}

@book {Natterer1986Mathematics,
    AUTHOR = {Natterer, F.},
     TITLE = {The mathematics of computerized tomography},
 PUBLISHER = {B. G. Teubner, Stuttgart; John Wiley \& Sons, Ltd.,
              Chichester},
      YEAR = {1986},
     PAGES = {x+222},
      ISBN = {3-519-02103-X},
   MRCLASS = {44A15 (92A07)},
  MRNUMBER = {856916},
MRREVIEWER = {L.\ A.\ Shepp},
}

@article {Rohlin1949Fundamental,
    AUTHOR = {Rohlin, V. A.},
     TITLE = {On the fundamental ideas of measure theory},
   JOURNAL = {Mat. Sbornik N.S.},
  FJOURNAL = {Mat. Sbornik N.S.},
    VOLUME = {25/67},
      YEAR = {1949},
     PAGES = {107--150},
   MRCLASS = {27.2X},
  MRNUMBER = {30584},
MRREVIEWER = {P.\ R.\ Halmos},
}

@article {Rohlin1949Decomposition,
    AUTHOR = {Rohlin, V. A.},
     TITLE = {On the decomposition of a dynamical system into transitive
              components},
   JOURNAL = {Mat. Sbornik N.S.},
    VOLUME = {25(67)},
      YEAR = {1949},
     PAGES = {235--249},
   MRCLASS = {46.3X},
  MRNUMBER = {0032958},
MRREVIEWER = {P. R. Halmos},
}

@book {Sarason1994Sub-Hardy,
    AUTHOR = {Sarason, Donald},
     TITLE = {Sub-Hardy Hilbert spaces in the unit disk},
    SERIES = {University of Arkansas Lecture Notes in the Mathematical
              Sciences},
    VOLUME = {10},
      NOTE = {A Wiley-Interscience Publication},
 PUBLISHER = {John Wiley \& Sons, Inc., New York},
      YEAR = {1994},
     PAGES = {xvi+95},
      ISBN = {0-471-04897-6},
   MRCLASS = {46E22 (30D55 30H05 46E20 47A45 47B35 47B38)},
  MRNUMBER = {1289670},
MRREVIEWER = {Thomas\ Kriete},
}

@article {Li2001Pseudo-duals,
    AUTHOR = {Li, Shidong and Ogawa, Hidemitsu},
     TITLE = {Pseudo-duals of frames with applications},
   JOURNAL = {Appl. Comput. Harmon. Anal.},
  FJOURNAL = {Applied and Computational Harmonic Analysis. Time-Frequency
              and Time-Scale Analysis, Wavelets, Numerical Algorithms, and
              Applications},
    VOLUME = {11},
      YEAR = {2001},
    NUMBER = {2},
     PAGES = {289--304},
      ISSN = {1063-5203,1096-603X},
   MRCLASS = {46B15 (42C15 46C05)},
  MRNUMBER = {1848709},
MRREVIEWER = {Ole\ Christensen},
       DOI = {10.1006/acha.2001.0347},
       URL = {https://doi.org/10.1006/acha.2001.0347},
}

@article {Hutchinson1981Fractals,
    AUTHOR = {Hutchinson, John E.},
     TITLE = {Fractals and self-similarity},
   JOURNAL = {Indiana Univ. Math. J.},
  FJOURNAL = {Indiana University Mathematics Journal},
    VOLUME = {30},
      YEAR = {1981},
    NUMBER = {5},
     PAGES = {713--747},
      ISSN = {0022-2518,1943-5258},
   MRCLASS = {49F20 (00A69 28A12 58C27)},
  MRNUMBER = {625600},
MRREVIEWER = {F.\ J.\ Almgren, Jr.},
       DOI = {10.1512/iumj.1981.30.30055},
       URL = {https://doi.org/10.1512/iumj.1981.30.30055},
}

@article {naimark1974direct,
    AUTHOR = {Naimark, M. A.},
     TITLE = {The direct integral of pairs of dual spaces},
   JOURNAL = {Dokl. Akad. Nauk SSSR},
  FJOURNAL = {Doklady Akademii Nauk SSSR},
    VOLUME = {217},
      YEAR = {1974},
     PAGES = {762--765},
      ISSN = {0002-3264},
   MRCLASS = {46L10 (47B99)},
  MRNUMBER = {0365164},
MRREVIEWER = {J. M. G. Fell},
}

@book {nielsen1980direct,
    AUTHOR = {Nielsen, Ole A.},
     TITLE = {Direct integral theory},
    SERIES = {Lecture Notes in Pure and Applied Mathematics},
    VOLUME = {61},
PUBLISHER = {Marcel Dekker, Inc., New York},
      YEAR = {1980},
     PAGES = {ix+165},
      ISBN = {0-8247-6971-6},
   MRCLASS = {46L10 (22D10 28B99)},
  MRNUMBER = {591683},
MRREVIEWER = {M. Takesaki},
}

@article {Pol93,
    AUTHOR = {Poltoratski, A. G.},
     TITLE = {Boundary behavior of pseudocontinuable functions},
   JOURNAL = {Algebra i Analiz},
  FJOURNAL = {Rossi\u\i skaya Akademiya Nauk. Algebra i Analiz},
    VOLUME = {5},
      YEAR = {1993},
    NUMBER = {2},
     PAGES = {189--210},
      ISSN = {0234-0852},
   MRCLASS = {30D55},
  MRNUMBER = {1223178 (94k:30090)},
MRREVIEWER = {D. Sarason},
    NOTE = {English translation in St. Petersburg Math. 5:2 (1994): 389--406.}
}

@book {nikolski1986treatise,
    AUTHOR = {Nikolski, N. K.},
     TITLE = {Treatise on the shift operator},
    SERIES = {Grundlehren der mathematischen Wissenschaften [Fundamental
              Principles of Mathematical Sciences]},
    VOLUME = {273},
      NOTE = {Spectral function theory,
              With an appendix by S. V. Hru\v{s}\v{c}ev [S. V.
              Khrushch\"{e}v] and V. V. Peller,
              Translated from the Russian by Jaak Peetre},
 PUBLISHER = {Springer-Verlag, Berlin},
      YEAR = {1986},
     PAGES = {xii+491},
      ISBN = {3-540-15021-8},
   MRCLASS = {47B37 (30H05 47-02 47B35)},
  MRNUMBER = {827223},
MRREVIEWER = {James\ Rovnyak},
       DOI = {10.1007/978-3-642-70151-1},
       URL = {https://doi.org/10.1007/978-3-642-70151-1},
}

@book {cima2000backward,
    AUTHOR = {Cima, Joseph A. and Ross, William T.},
     TITLE = {The backward shift on the Hardy space},
    SERIES = {Mathematical Surveys and Monographs},
    VOLUME = {79},
 PUBLISHER = {American Mathematical Society, Providence, RI},
      YEAR = {2000},
     PAGES = {xii+199},
      ISBN = {0-8218-2083-4},
   MRCLASS = {47B38 (30D55 46E15 47A15 47A16)},
  MRNUMBER = {1761913},
MRREVIEWER = {Harold\ S.\ Shapiro},
       DOI = {10.1090/surv/079},
       URL = {https://doi.org/10.1090/surv/079},
}

@article {Clark72,
    AUTHOR = {Clark, Douglas N.},
     TITLE = {One dimensional perturbations of restricted shifts},
   JOURNAL = {J. Analyse Math.},
  FJOURNAL = {Journal d'Analyse Math 'ematique},
    VOLUME = {25},
      YEAR = {1972},
     PAGES = {169--191},
      ISSN = {0021-7670},
   MRCLASS = {47A55},
  MRNUMBER = {0301534 (46 \#692)},
MRREVIEWER = {S. R. Caradus},
}

@article {Aleks89a,
    AUTHOR = {Aleksandrov, A. B.},
     TITLE = {Inner functions and related spaces of pseudocontinuable
              functions},
   JOURNAL = {Zap. Nauchn. Sem. Leningrad. Otdel. Mat. Inst. Steklov.
              (LOMI)},
  FJOURNAL = {Zapiski Nauchnykh Seminarov Leningradskogo Otdeleniya
              Matematicheskogo Instituta imeni V. A. Steklova Akademii Nauk
              SSSR (LOMI)},
    VOLUME = {170},
      YEAR = {1989},
     PAGES = {7--33, 321},
      ISSN = {0373-2703},
   MRCLASS = {30D55 (30H05 46E30 47B38)},
  MRNUMBER = {1039571},
MRREVIEWER = {D.\ Sarason},
       DOI = {10.1007/BF01099304},
       URL = {https://doi.org/10.1007/BF01099304},
}

@article {dB61a,
    AUTHOR = {de Branges, Louis},
     TITLE = {Some Hilbert spaces of entire functions. II},
   JOURNAL = {Trans. Amer. Math. Soc.},
  FJOURNAL = {Transactions of the American Mathematical Society},
    VOLUME = {99},
      YEAR = {1961},
     PAGES = {118--152},
      ISSN = {0002-9947,1088-6850},
   MRCLASS = {30.85},
  MRNUMBER = {133456},
       DOI = {10.2307/1993448},
       URL = {https://doi.org/10.2307/1993448},
}

@incollection {bezuglyi2019graph,
    AUTHOR = {Bezuglyi, Sergey and Jorgensen, Palle E. T.},
     TITLE = {Graph Laplace and Markov operators on a measure space},
BOOKTITLE = {Linear systems, signal processing and hypercomplex analysis},
    SERIES = {Oper. Theory Adv. Appl.},
    VOLUME = {275},
     PAGES = {67--138},
PUBLISHER = {Birkh\"{a}user/Springer, Cham},
      YEAR = {2019},
   MRCLASS = {60J05 (05C80 31C20 37L30 47B38 60J45)},
  MRNUMBER = {3968033},
MRREVIEWER = {Wojciech Bartoszek},
}

@incollection {bezuglyi2021symmetric,
    AUTHOR = {Bezuglyi, Sergey and Jorgensen, Palle E. T.},
     TITLE = {Symmetric measures, continuous networks, and dynamics},
BOOKTITLE = {New directions in function theory: from complex to
              hypercomplex to non-commutative},
    SERIES = {Oper. Theory Adv. Appl.},
    VOLUME = {286},
     PAGES = {139--197},
PUBLISHER = {Birkh\"{a}user/Springer, Cham},
      YEAR = {2021},
   MRCLASS = {60J25 (37L40 47N30 60G20)},
  MRNUMBER = {4368036},
       DOI = {10.1007/978-3-030-76473-9\_6},
       URL = {https://doi-org.proxy.lib.uiowa.edu/10.1007/978-3-030-76473-9_6},
}

@article{chang1997conditioning,
  title={Conditioning as disintegration},
  author={Chang, Joseph T and Pollard, David},
  journal={Statistica Neerlandica},
  volume={51},
  number={3},
  pages={287--317},
  year={1997},
  publisher={Wiley Online Library}
}

@incollection{herr2020harmonic,
  author = {Herr, John E. and Jorgensen, Palle E. T. and Weber, Eric S.},
  editor    = {Ruiz, P. and Chen, J. and Rogers, L. and Strichartz, R. and Teplyaev, A.},
  booktitle     = {Analysis, Probability and Mathematical Physics on Fractals},
  title   = {Harmonic Analysis of Fractal Measures:  Basis and Frame Algorithms},
  publisher = {World Scientific},
  year      = {2020},
  volume    = {5},
  series    = {Fractals and Dynamics in  Mathematics, Science, and the Arts},
  pages     = {163--221},
  doi = {10.1142/9789811215537_0005},
}

@article{Kacz37,
    AUTHOR = {Kaczmarz, Stefan},
     TITLE = {Angen "aherte Aufl" osung von Systemen linearer Gleichungen},
   JOURNAL = {Bulletin International de l'Acad\'{e}mie Plonaise des Sciences et des Lettres. Classe des Sciences Math\'{e}matiques et Naturelles. S\'{e}rie A. Sciences Math\'{e}matiques},
  FJOURNAL = {Bulletin International de l'Acad\'{e}mie Plonaise des Sciences et des Lettres. Classe des Sciences Math\'{e}matiques et Naturelles. S\'{e}rie A. Sciences Math\'{e}matiques},
    VOLUME = {35},
      YEAR = {1937},
     PAGES = {355--357},
}

@article {kakutani1948equivalence,
    AUTHOR = {Kakutani, Shizuo},
     TITLE = {On equivalence of infinite product measures},
   JOURNAL = {Ann. of Math. (2)},
  FJOURNAL = {Annals of Mathematics. Second Series},
    VOLUME = {49},
      YEAR = {1948},
     PAGES = {214--224},
      ISSN = {0003-486X},
   MRCLASS = {27.2X},
  MRNUMBER = {23331},
MRREVIEWER = {B. Jessen},
       DOI = {10.2307/1969123},
       URL = {https://doi.org/10.2307/1969123},
}

@article {DL14a,
    AUTHOR = {Dutkay, Dorin Ervin and Lai, Chun-Kit},
     TITLE = {Uniformity of measures with {F}ourier frames},
   JOURNAL = {Adv. Math.},
  FJOURNAL = {Advances in Mathematics},
    VOLUME = {252},
      YEAR = {2014},
     PAGES = {684--707},
      ISSN = {0001-8708},
   MRCLASS = {28A80 (42C15)},
  MRNUMBER = {3144246},
MRREVIEWER = {Hua Qiu},
       DOI = {10.1016/j.aim.2013.11.012},
       URL = {http://dx.doi.org/10.1016/j.aim.2013.11.012},
}
\end{document}